\documentclass[twoside, 11pt]{article}
\usepackage[top=.8
in, bottom=1in, left=1in, right=1in]{geometry}

\usepackage{enumerate,float,mathtools}

\usepackage{color, section, amsthm, textcase, setspace, amssymb, lineno, 
amsmath, amssymb, amsfonts, latexsym, fancyhdr, longtable, ulem}

\usepackage{tikz,tikz-3dplot,tkz-graph}
\usetikzlibrary{decorations.markings}
\usetikzlibrary{arrows.meta}

\newtheorem{theorem}{Theorem}[section]
\newtheorem{lemma}[theorem]{Lemma}
\newtheorem{corollary}[theorem]{Corollary}

\theoremstyle{definition}

\newtheorem{example}[theorem]{Example}
\newtheorem{remark}[theorem]{Remark}
\newtheorem{case}{Case}

\newcommand{\Complex}{\mathbb C} 

\newcommand{\mf}{\mathfrak}

\newcommand{\ul}{\underline}
\newcommand{\Cn}{C_{\leq n}}
\newcommand{\dd}{\hspace{.1cm}|\hspace{.1cm}}
\newcommand{\ind}{{\rm ind \hspace{.1cm}}}
\newcommand{\rk}{{\rm rk \hspace{.1cm}}}
\newcommand{\A}{{\rm A}}
\newcommand{\C}{{\rm C}}
\newcommand{\ad}{{\rm ad \hspace{.05cm}}}
\newcommand{\lc}{\left\lceil}
\newcommand{\rc}{\right\rceil}
\newcommand{\lf}{\left\lfloor}
\newcommand{\rf}{\right\rfloor}

\newcommand{\dk}{\rm DK}

\newcommand{\sright}[1]{\sigma(\overrightarrow{#1})}
\newcommand{\sleft}[1]{\sigma(\overleftarrow{#1})}

\newcommand{\ceil}[1]{\left \lceil #1 \right \rceil }

\newcommand{\mir}{\rm mir}
\newcommand{\tright}[1]{\tau(\overrightarrow{#1})}
\newcommand{\tleft}[1]{\tau(\overleftarrow{#1})}

\begin{document}

\title{\bf The unbroken spectrum of Frobenius seaweeds II:  type-B and type-C}
\author{Alex Cameron$^*$, Vincent E. Coll, Jr.$^{**}$, Matthew Hyatt$^{\dagger}$, and Colton Magnant$^{\dagger\dagger}$ }
\maketitle

\noindent
\textit{$^*$Department of Mathematics, Lehigh University, Bethlehem, PA, USA: akc214@lehigh.edu }\\
\textit{$^{**}$Department of Mathematics, Lehigh University, Bethlehem, PA, USA:  vec208@lehigh.edu}\\
\textit{$^{\dagger}$FactSet Research Systems, New York, NY, USA:  matthewdhyatt@gmail.com }\\
\textit{$^{\dagger\dagger}$
Department of Mathematics, Clayton State University, Morrow, GA, USA: dr.colton.magnant@gmail.com }


\begin{abstract}
\noindent
Analogous to the Type-$A_{n-1}=\mathfrak{sl}(n)$ case, we show that if $\mathfrak{g}$ is a Frobenius seaweed subalgebra of $B_{n}=\mathfrak{so}(2n+1)$ or
$C_{n}=\mathfrak{sp}(2n)$, then the spectrum of the adjoint of a principal element consists of an unbroken set of integers whose multiplicities have a symmetric distribution.    
\end{abstract}



\section{Introduction}\label{sec:intro}
\textit{Notation: } All Lie algebras will be finite dimensional over $\Complex$, and the Lie multiplication will be denoted by [-,-].

\bigskip
The \textit{index} of a Lie algebra is an important algebraic invariant and, for \textit{seaweed algebras}, is bounded by the algebra's rank: $ \ind \mf{g} \leq \rk \mf{g}$, (see \textbf{\cite{dk}}). More formally, the index of a Lie algebra $\mf{g}$ is given by 

\[\ind \mf{g}=\min_{F\in \mf{g^*}} \dim  (\ker (B_F)),\]

\noindent where $F$ is a linear form on $\mf{g}$, and $B_F$ is the associated skew-symmetric bilinear \textit{Kirillov form}, defined by $B_F(x,y)=F([x,y])$ for $x,y\in \mf{g}$.  On a given $\mf{g}$, index-realizing functionals are called \textit{regular} and exist in profusion, being dense in both the Zariski and Euclidean topologies of $\mf{g}^*$. 

Of particular interest are Lie algebras which have index zero. Such algebras are called \textit{Frobenius}
and have been studied extensively from the point of view of invariant theory \textbf{\cite{Ooms1}} and are of special interest in deformation and quantum group theory stemming from their connection with the classical Yang-Baxter equation (see \textbf{\cite{G1}} and \textbf{\cite{G2}}).
A regular functional $F$ on a Frobenius Lie algebra $\mathfrak{g}$ is called a \textit{Frobenius functional}; equivalently, $B_F(-,-)$ is non-degenerate.  Suppose $B_F(-,-)$ is non-degenerate and let $[F]$ be the matrix of $B_F(-,-)$ relative to some basis 
$\{x_1,\dots,x_n  \}$ of $\mathfrak{g}$.  In \textbf{\cite{BD}}, Belavin and Drinfeld showed that   

\[
\sum_{i,j}[F]^{-1}_{ij}x_i\wedge x_j
\]

\noindent
is the infinitesimal of a \textit{Universal Deformation Formula} (UDF) based on $\mathfrak{g}$.  A UDF based on $\mathfrak{g}$ can be used to deform the universal enveloping algebra of $\mathfrak{g}$ and also the function space of any Lie group which contains $\mathfrak{g}$ in its Lie algebra of derivations.
Despite the existence proof of Belavin and Drinfel'd, 
only two UDF's are known:  the exponential and quasi-exponential.  These are based, respectively, on the abelian \textbf{\cite{UDF1}} and non-abelian \textbf{\cite{UDF2}} Lie algebras of dimension two (see also \textbf{\cite{Twist}}).

A Frobenius functional can be algorithmically produced as a by-product of the Kostant Cascade (see \textbf{\cite{Joseph4} and \cite{Kostant}}).  
If $F$ is a Frobenius functional on $\mathfrak{g}$, then the natural map $\mathfrak{g} \rightarrow \mathfrak{g}^*$ defined by $x \mapsto  F([x,-])$ is an isomorphism.  The image of $F$ under the inverse of this map is called a \textit{principal element} of $\mathfrak{g}$ and will be denoted $\widehat{F}$.  It is the unique element of $\mathfrak{g}$ such that 

$$
F\circ \ad \widehat{F}= F([\widehat{F},-]) = F.  
$$

As a consequence of Proposition 3.1 in \textbf{\cite{Ooms2}}, Ooms established that the spectrum of the adjoint of a principal element of a Frobenius Lie algebra is independent of the principal element chosen to compute it (see also \textbf{\cite{G2}}, Theorem 3).  Generally, the eigenvalues of ad$\widehat{F}$ can take on virtually any value (see \textbf{\cite{DIATTA}} for examples). But, in their formal study of principal elements \textbf{\cite{G3}}, Gerstenhaber and Giaquinto showed that if $\mathfrak{g}$ is a Frobenius seaweed subalgebra of $A_{n-1}=\mathfrak{sl}(n)$, then the spectrum of the adjoint of a principal element of $\mathfrak{g}$ consists entirely of integers.\footnote{Joseph, seemingly unaware of the type-A result of \textbf{{\cite{G3}}} used different methods to strongly extend this integrity result to all seaweed subalgebras of semisimple Lie algebras \textbf{\cite{Joseph2}}.}  Subsequently, the last three of the current authors showed that this spectrum must actually be an \textit{unbroken} sequence of integers centered at one half \textbf{\cite{Coll typea}}.\footnote{The paper of Coll et al. appears as a folllow-up to the \textit{Lett. in Math. Physics} article 
by Gerstenhaber and Giaquinto \textbf{\cite{G3}}, where they claim that the eigenvalues of the adjoint representation of a semisimple principal element of a Frobenius seaweed subalgebra of $\mathfrak{sl}(n)$ consists of an unbroken sequence of integers. However, M. Dufflo, in a private communication to those authors, indicated that their proof contained an error.} Moreover, the dimensions of the associated eigenspaces are shown to have a symmetric distribution.  

The goal of this paper is to establish the following theorem, which asserts that the above-described spectral phenomena for type A is also exhibited in the classical type-B and type-C cases.

\begin{theorem}\label{thm:main}
If $\mathfrak{g}$ is a Frobenius seaweed subalgebra of type B or type C $($i.e., $\mathfrak{g}$ is a subalgebra of $\mathfrak{so}(2n+1)$ or $\mathfrak{sp}(2n)$, respectively$)$ and $\widehat{F}$ is a principal element of $\mathfrak{g}$, then the spectrum of $\ad \widehat{F}$ consists of an unbroken set of integers centered at one-half.  Moreover, the dimensions of the associated eigenspaces form a symmetric distribution. 
\end{theorem}

\noindent
\begin{remark}
In the first article of this series \textbf{\cite{Coll typea}}, the type-A unbroken symmetric spectrum result is established using a combinatorial argument based on the graph-theoretic meander construction of Dergachev and Kirillov \textbf{\cite{dk}}.  Here, the combinatorial arguments  heavily leverage the results of \textbf{\cite{Coll typea}}, but the inductions are predicated on the more sophisticated ``orbit meander" construction of Joseph \textbf{\cite{Joseph2}}.

\end{remark}

\section{Seaweeds}

Let $\mf{g}$ be a simple Lie algebra equipped with a triangular decomposition 
\begin{eqnarray} \label{triangular}
\mf{g}=\mf{u_+}\oplus\mf{h}\oplus\mf{u_-},
\end{eqnarray}
where $\mf{h}$ is a Cartan subalgebra of $\mf{g}$. Let $\Delta$ be its root system 
where $\Delta_{+}$ are the \textit{positive roots} on $\mf{u_+}$ and 
$\Delta_-$ are the \textit{negative roots} roots
on $\mf{u_-}$, and let $\Pi$ denote the set of \textit{simple roots} of $\mf{g}$. 
Given $\beta\in\Delta$,
let $\mf{g}_{\beta}$ denote its corresponding root space,
and let $x_\beta$ denote the element of weight $\alpha$ in a 
Chevalley basis of $\mf{g}$. Given a subset 
$\pi_1\subseteq \Pi$, let $\mf{p}_{\pi_1}$ denote the parabolic subalgebra of
$\mf{g}$ generated by all $\mf{g}_{\beta}$ such that $-\beta\in\Pi$ or $\beta\in\pi_1$.
Such a parabolic subalgebra is called \textit{standard}  with respect to the Borel subalgebra 
$\mf{u_-}\oplus\mf{h}$, and it is known that every parabolic subalgebra is conjugate to exactly
one standard parabolic subalgebra.

Formation of a seaweed subalgebra of $\mathfrak{g}$ requires two weakly opposite parabolic
subalgebras, i.e., two parabolic subalgebras $\mf{p}$ and $\mf{p'}$ such that 
$\mf{p} + \mf{p'}= \mf{g}$. In this case,  $\mf{p}\cap\mf{p'}$ is called a \textit{seaweed}, or in the nomenclature of Joseph \textbf{\cite{Joseph2}}, a  \textit{biparabolic} subalgebra of $\mf{g}$.  Given a subset $\pi_2\subseteq \Pi$,
let $\mf{p}_{\pi_2}^-$ denote the parabolic subalgebra of
$\mf{g}$ generated by all $\mf{g}_{\beta}$ such that $\beta\in\Pi$ 
or $-\beta\in\pi_2$.
Given two subsets $\pi_1,\pi_2\subseteq\Pi$, we have 
$\mf{p}_{\pi_1}+\mf{p}_{\pi_2}=\mf{g}$.  We now define the seaweed

\begin{eqnarray*}
\mf{p}(\pi_1\dd \pi_2)=\mf{p}_{\pi_1}\cap\mf{p}_{\pi_2}^-,
 \end{eqnarray*}

\noindent
which is said to be \textit{standard} with respect to the triangular decomposition in (\ref{triangular}).
Any seaweed is conjugate to a standard one, so it suffices to work
with standard seaweeds only. Note that an arbitrary seaweed may be conjugate to more than one standard seaweed (see \textbf{\cite{Panyushev1}}). 

We will often assume that $\pi_1\cup\pi_2=\Pi$,
for if not then $\mf{p}(\pi_1\dd \pi_2)$ can be expressed
as a direct sum of seaweeds.
Additionally, we use superscripts and subscripts 
to specify the type and rank of the containing simple Lie algebra 
$\mf{g}$. For example $\mf{p}^{\C}_{n}(\pi_1\dd \pi_2)$ is a 
seaweed subalgebra of $C_n=\mathfrak{sp}(2n)$, the symplectic Lie algebra of rank $n$.

It will be convenient to visualize the simple roots of a seaweed 
by constructing a graph, which we call 
a \textit{split Dynkin diagram}.  
Suppose $\mf{g}$ has rank $n$, and let $\Pi=\{\alpha_n,\dots ,\alpha_1\}$,
where $\alpha_1$ is the exceptional root for types B and C. 
Draw two horizontal lines of $n$ vertices, say $v_n^+,\dots ,v_1^+$ on top
and $v_n^-,\dots ,v_1^-$ on the bottom.
Color $v_i^+$ black if $\alpha_i\in\pi_1$, color $v_i^-$ black if $\alpha_i\in\pi_2$,
and color all other vertices white. Furthermore, if $\alpha_i,\alpha_j\in\pi_1$ are
not orthogonal, connect $v_i^+$ and $v_j^+$ with an edge in the 
standard way used in Dynkin diagrams.
Do the same for bottom vertices according to the roots in $\pi_2$. A \textit{maximally connected component} of a split Dynkin diagram is defined in the obvious way, and such a component is of type B or C if it contains the exceptional root $\alpha_1$; otherwise the component is of type A.  See Figure \ref{sdynkin}.
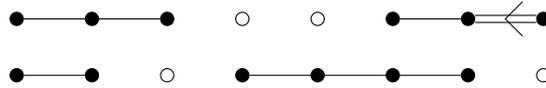
\begin{figure}[H]
\[\begin{tikzpicture}
[decoration={markings,mark=at position 0.6 with 
{\arrow{angle 90}{>}}}]

\draw (1,.75) node[draw,circle,fill=black,minimum size=5pt,inner sep=0pt] (1+) {};
\draw (2,.75) node[draw,circle,fill=black,minimum size=5pt,inner sep=0pt] (2+) {};
\draw (3,.75) node[draw,circle,fill=black,minimum size=5pt,inner sep=0pt] (3+) {};
\draw (4,.75) node[draw,circle,fill=white,minimum size=5pt,inner sep=0pt] (4+) {};
\draw (5,.75) node[draw,circle,fill=white,minimum size=5pt,inner sep=0pt] (5+) {};
\draw (6,.75) node[draw,circle,fill=black,minimum size=5pt,inner sep=0pt] (6+) {};
\draw (7,.75) node[draw,circle,fill=black,minimum size=5pt,inner sep=0pt] (7+) {};
\draw (8,.75) node[draw,circle,fill=black,minimum size=5pt,inner sep=0pt] (8+) {};

\draw (1,0) node[draw,circle,fill=black,minimum size=5pt,inner sep=0pt] (1-) {};
\draw (2,0) node[draw,circle,fill=black,minimum size=5pt,inner sep=0pt] (2-) {};
\draw (3,0) node[draw,circle,fill=white,minimum size=5pt,inner sep=0pt] (3-) {};
\draw (4,0) node[draw,circle,fill=black,minimum size=5pt,inner sep=0pt] (4-) {};
\draw (5,0) node[draw,circle,fill=black,minimum size=5pt,inner sep=0pt] (5-) {};
\draw (6,0) node[draw,circle,fill=black,minimum size=5pt,inner sep=0pt] (6-) {};
\draw (7,0) node[draw,circle,fill=black,minimum size=5pt,inner sep=0pt] (7-) {};
\draw (8,0) node[draw,circle,fill=white,minimum size=5pt,inner sep=0pt] (8-) {};

\draw (1-) to (2-);
\draw (4-) to (7-);
\draw (1+) to (3+);
\draw (6+) to (7+);
\draw [double distance=.8mm,postaction={decorate}] (8+) to (7+);

;\end{tikzpicture}\]
\caption{The split Dynkin diagram of $\mf{p}_8^\C(
\{\alpha_8,\alpha_7,\alpha_6,\alpha_3, \alpha_2, \alpha_1
\dd \alpha_8,\alpha_7,\alpha_5,
\alpha_4,\alpha_3,\alpha_2\})$}
\label{sdynkin}
\end{figure}



Given a seaweed $\mf{p}(\pi_1\dd\pi_2)$ of a simple 
Lie algbera $\mf{g}$, let $W$ denote the Weyl group of its
root system $\Delta$, generated by the reflections $s_\alpha$ 
such that $\alpha\in\Pi$. For $j=1,2$ define $W_{\pi_j}$ 
to be the subgroup of $W$ generated 
by $s_\alpha$ such that $\alpha\in\pi_j$. Let $w_j$
denote the unique longest (in the usual Coxeter sense) 
element of $W_{\pi_j}$,
and define an action
\[i_j\alpha=
\begin{cases}-w_j\alpha, & \text{ for all }\alpha\in\pi_j;\\
\alpha, & \text{ for all }\alpha\in\Pi\setminus\pi_j.
\end{cases}\]
Note that $i_j$ is an involution. 
For components of types B and C, the longest element $w_j = -id$.

To visualize the action of $i_j$, we append dashed edges
to the split Dynkin diagram of a seaweed. Specifically,
we draw a dashed edge from $v_i^+$ to $v_j^+$
if $i_1\alpha_i=\alpha_j$, and we draw a dashed  
edge from $v_i^-$ to $v_j^-$ if $i_2\alpha_i=\alpha_j$.
For simplicity, we will omit drawing the looped edge in the case
that $i_j\alpha_i=\alpha_i$.
We call the resulting graph the \textit{orbit meander}
of the associated seaweed.
See Figure \ref{fig:orbit meander}.
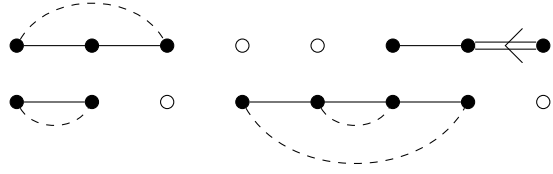
\begin{figure}[H]
\[\begin{tikzpicture}
[decoration={markings,mark=at position 0.6 with 
{\arrow{angle 90}{>}}}]

\draw (1,.75) node[draw,circle,fill=black,minimum size=5pt,inner sep=0pt] (1+) {};
\draw (2,.75) node[draw,circle,fill=black,minimum size=5pt,inner sep=0pt] (2+) {};
\draw (3,.75) node[draw,circle,fill=black,minimum size=5pt,inner sep=0pt] (3+) {};
\draw (4,.75) node[draw,circle,fill=white,minimum size=5pt,inner sep=0pt] (4+) {};
\draw (5,.75) node[draw,circle,fill=white,minimum size=5pt,inner sep=0pt] (5+) {};
\draw (6,.75) node[draw,circle,fill=black,minimum size=5pt,inner sep=0pt] (6+) {};
\draw (7,.75) node[draw,circle,fill=black,minimum size=5pt,inner sep=0pt] (7+) {};
\draw (8,.75) node[draw,circle,fill=black,minimum size=5pt,inner sep=0pt] (8+) {};

\draw (1,0) node[draw,circle,fill=black,minimum size=5pt,inner sep=0pt] (1-) {};
\draw (2,0) node[draw,circle,fill=black,minimum size=5pt,inner sep=0pt] (2-) {};
\draw (3,0) node[draw,circle,fill=white,minimum size=5pt,inner sep=0pt] (3-) {};
\draw (4,0) node[draw,circle,fill=black,minimum size=5pt,inner sep=0pt] (4-) {};
\draw (5,0) node[draw,circle,fill=black,minimum size=5pt,inner sep=0pt] (5-) {};
\draw (6,0) node[draw,circle,fill=black,minimum size=5pt,inner sep=0pt] (6-) {};
\draw (7,0) node[draw,circle,fill=black,minimum size=5pt,inner sep=0pt] (7-) {};
\draw (8,0) node[draw,circle,fill=white,minimum size=5pt,inner sep=0pt] (8-) {};

\draw (1-) to (2-);
\draw (4-) to (7-);
\draw (1+) to (3+);
\draw (6+) to (7+);
\draw [double distance=.8mm,postaction={decorate}] (8+) to (7+);

\draw [dashed] (1+) to [bend left=60] (3+);
\draw [dashed] (1-) to [bend right=60] (2-);
\draw [dashed] (4-) to [bend right=60] (7-);
\draw [dashed] (5-) to [bend right=60] (6-);

;\end{tikzpicture}\]
\caption{The orbit meander of $\mf{p}_8^\C(
\{\alpha_8,\alpha_7,\alpha_6,\alpha_3, \alpha_2, \alpha_1
\dd \alpha_8,\alpha_7,\alpha_5,
\alpha_4,\alpha_3,\alpha_2\})$}
\label{fig:orbit meander}
\end{figure}

It turns out that the index of $\mf{p}(\pi_1\dd\pi_2)$
is governed by the orbits of the cyclic group $<i_1i_2>$
acting on $\Pi$ (see Section 7.16 in \textbf{\cite{Joseph}}).  Since we are interested in only Frobenius seaweeds, we do not require the full power of that theorem, needing only the following corollary. 


\begin{theorem}[Joseph \textbf{\cite{Joseph}}, Section 7.16]\label{thm:frobenius}
Given subsets $\pi_1,\pi_2\subseteq\Pi$ such that $\pi_1\cup\pi_2=\Pi$, 
let $\pi_{\cup}=\Pi\setminus (\pi_1\cap\pi_2)$.
The seaweed $\mf{p}(\pi_1\dd\pi_2)$ is Frobenius if and only if
every $<i_1i_2>$ orbit contains 
exactly one element from $\pi_{\cup}$.
\end{theorem}

\begin{example} (Frobenius seaweed)
The seaweed from Figure \ref{fig:orbit meander} is Frobenius as the $<i_1i_2>$ orbits of $\mf{p}_8^\C(
\{\alpha_8,\alpha_7,\alpha_6,\alpha_3, \alpha_2, \alpha_1
\dd \alpha_8,\alpha_7,\alpha_5,
\alpha_4,\alpha_3,\alpha_2\})$ are
$\{\alpha_7,\alpha_6,\alpha_8\}$, $\{\alpha_5,\alpha_2\}$, $\{\alpha_4,\alpha_3\}$, $\{\alpha_1\}$, and each orbit contains exactly one element
from $\pi_{\cup}=\{\alpha_6,\alpha_5,\alpha_4,\alpha_1\}$.
\end{example}

\begin{example} 
For an example of a seaweed that is not Frobenius, consider
the orbit meander of $\mf{p}_7^\A(\{\alpha_7,\alpha_6,\alpha_5,
\alpha_4,\alpha_3,\alpha_2
\dd \alpha_7,\alpha_6,\alpha_4,\alpha_3,\alpha_2,\alpha_1\})$ shown in Figure \ref{fig:orbit meander nonfrob} below.
The $<i_1i_2>$ orbits $\{\alpha_7,\alpha_3\}$ and 
$\{\alpha_6,\alpha_2\}$ contain no elements from
$\pi_{\cup} = \{\alpha_5, \alpha_1\}$, and the orbit $\{\alpha_5,\alpha_4,\alpha_1\}$ contains
two elements from $\pi_{\cup}$.

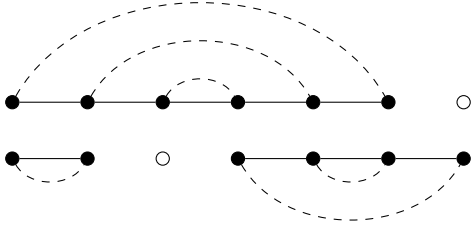
\begin{figure}[H]
\[\begin{tikzpicture}
[decoration={markings,mark=at position 0.6 with 
{\arrow{angle 90}{>}}}]

\draw (1,.75) node[draw,circle,fill=black,minimum size=5pt,inner sep=0pt] (1+) {};
\draw (2,.75) node[draw,circle,fill=black,minimum size=5pt,inner sep=0pt] (2+) {};
\draw (3,.75) node[draw,circle,fill=black,minimum size=5pt,inner sep=0pt] (3+) {};
\draw (4,.75) node[draw,circle,fill=black,minimum size=5pt,inner sep=0pt] (4+) {};
\draw (5,.75) node[draw,circle,fill=black,minimum size=5pt,inner sep=0pt] (5+) {};
\draw (6,.75) node[draw,circle,fill=black,minimum size=5pt,inner sep=0pt] (6+) {};
\draw (7,.75) node[draw,circle,fill=white,minimum size=5pt,inner sep=0pt] (7+) {};

\draw (1,0) node[draw,circle,fill=black,minimum size=5pt,inner sep=0pt] (1-) {};
\draw (2,0) node[draw,circle,fill=black,minimum size=5pt,inner sep=0pt] (2-) {};
\draw (3,0) node[draw,circle,fill=white,minimum size=5pt,inner sep=0pt] (3-) {};
\draw (4,0) node[draw,circle,fill=black,minimum size=5pt,inner sep=0pt] (4-) {};
\draw (5,0) node[draw,circle,fill=black,minimum size=5pt,inner sep=0pt] (5-) {};
\draw (6,0) node[draw,circle,fill=black,minimum size=5pt,inner sep=0pt] (6-) {};
\draw (7,0) node[draw,circle,fill=black,minimum size=5pt,inner sep=0pt] (7-) {};

\draw (1+) to (2+);
\draw (2+) to (3+);
\draw (3+) to (4+);
\draw (4+) to (5+);
\draw (5+) to (6+);
\draw (1-) to (2-);
\draw (4-) to (5-);
\draw (5-) to (6-);
\draw (6-) to (7-);

\draw [dashed] (1+) to [bend left=60] (6+);
\draw [dashed] (2+) to [bend left=60] (5+);
\draw [dashed] (3+) to [bend left=60] (4+);
\draw [dashed] (1-) to [bend right=60] (2-);
\draw [dashed] (4-) to [bend right=60] (7-);
\draw [dashed] (5-) to [bend right=60] (6-);

;\end{tikzpicture}\]
\caption{The orbit meander of $\mf{p}_7^\A(\{\alpha_7,\alpha_6,\alpha_5,
\alpha_4,\alpha_3,\alpha_2
\dd \alpha_7,\alpha_6,\alpha_4,\alpha_3,\alpha_2,\alpha_1\})$}
\label{fig:orbit meander nonfrob}
\end{figure}

\end{example}

\section{Principal Elements}

Given a Frobenius seaweed with Frobenius functional
$F$ and corresponding principal element $\widehat{F}$, the eigenvalues of
$\ad \widehat{F}$ are independent of which Frobenius functional is chosen \textbf{\cite{Ooms2}}.
We call these eigenvalues the \textit{spectrum} of the seaweed.
In this section, we describe an algorithm
for computing them.

Let $\mf{p}(\pi_1\dd\pi_2)$ be Frobenius and $F_{\pi_1,\pi_2}$ be an associated Frobenius functional with principal element $ \widehat{F}_{\pi_1,\pi_2}$.  (We will simply write $F$ and $\widehat{F}$ when the seaweed is understood.) 
Let $\sigma$ be a maximally connected component of either
$\pi_1$ or $\pi_2$, and for convenience let 
$\sigma=\{\alpha_k,\alpha_{k-1},\dots ,\alpha_1\}$ where $\alpha_1$ is the
exceptional root if $\sigma$ is of type B or C.


Each eigenvalue of $\ad \widehat{F}$ can be expressed as a linear combination
of elements $\alpha_i(\widehat{F})$ where $\alpha_i$ is a simple root. We call such numbers \textit{simple eigenvalues}.
In many cases, the simple eigenvalues are determined.

\begin{lemma}[Joseph \textbf{\cite{Joseph}}, Section 5]
\label{eigenvalue table}
In Table \ref{tab:simple eigenvalue} below, the given value is 
$\alpha_i(\widehat{F})$ if $\sigma$ is a maximally connected component 
of $\pi_1$, and it is $-\alpha_i(\widehat{F})$ if 
$\sigma$ is a maximally connected component of $\pi_2$.
In either case it is assumed that $\alpha_i\in\sigma$.
\end{lemma}

\begin{table}[H]
\[\begin{tabular}{|l|l|l|}
\hline
Type & $\pm\alpha_i(\widehat{F})$ & $\pm\alpha_i(\widehat{F})$ \\
\hline
\hline
$A_k:k\geq 1$ & 
1, if $i_j\alpha_i=\alpha_i$
& \\
\hline
$B_{2k-1}:k\geq 2$ & $(-1)^{i-1}$, if $1\leq i\leq 2k-1$ & \\
\hline
$B_{2k}:k\geq 2$ & $(-1)^{i}$, if $2\leq i\leq 2k$ & 0, if $i=1$\\
\hline
$C_{k}:k\geq 2$ & $0$, if $2\leq i\leq k$ & 1, if $i=1$\\
\hline
\end{tabular}\]
\caption{Values of $\pm\alpha_i(\widehat{F})$}
\label{tab:simple eigenvalue}
\end{table}

For the cases not covered by Table 1, that is, for maximally connected components of type $A_k$ with $k \geq 2$ and $i_j\alpha_i \neq \alpha_i$, the following lemma 
(which includes a corrected typo from \textbf{\cite{Joseph}})
can be applied.

\begin{lemma}[Joseph \textbf{\cite{Joseph}}, Section 5]
\label{eigenvalue orbit}
For each equation below, $\sigma$ is assumed to be a maximally 
connected component of $\pi_1$.
If $\sigma$ is a maximally connected component of $\pi_2$,
then replace $\alpha_i\mapsto -\alpha_i$ and $i_1\mapsto i_2$.

Let $\sigma$ be of type $A$, then
\[\alpha_i(\widehat{F})+i_1\alpha_i(\widehat{F})=
\begin{dcases}
1, & \text{ if }\sigma \text{ is of type }A_k\text{ and }
(\alpha_i,i_1\alpha_i)<0; \\
0, & \text{ if }\sigma \text{ is of type }A_k\text{ and }
(\alpha_i,i_1\alpha_i)=0.
\end{dcases}\]
\end{lemma}

Notice $\sigma$ contains $\alpha_i$ with $(\alpha_i,i_1\alpha_i)<0$ only when $\sigma$ is of type $A_{2k}$ with $k \geq 1$.  

\begin{example} Using Table \ref{tab:simple eigenvalue} and applying Lemma \ref{eigenvalue orbit}, we compute the simple eigenvalues of the seaweed $\mf{p}_8^\C(
\{\alpha_8,\alpha_7,\alpha_6,\alpha_3, \alpha_2, \alpha_1
\dd \alpha_8,\alpha_7,\alpha_5,
\alpha_4,\alpha_3,\alpha_2\})$. See Figure \ref{fig:simple eigenvalues} where the simple eigenvalues are noted above or below the appropriate vertex in the orbit meander for this seaweed.

\begin{figure}[H]
\[\begin{tikzpicture}
[decoration={markings,mark=at position 0.6 with 
{\arrow{angle 90}{>}}}]

\draw (1,.75) node[draw,circle,fill=black,minimum size=5pt,inner sep=0pt] (1+) {};
\draw (2,.75) node[draw,circle,fill=black,minimum size=5pt,inner sep=0pt] (2+) {};
\draw (3,.75) node[draw,circle,fill=black,minimum size=5pt,inner sep=0pt] (3+) {};
\draw (4,.75) node[draw,circle,fill=white,minimum size=5pt,inner sep=0pt] (4+) {};
\draw (5,.75) node[draw,circle,fill=white,minimum size=5pt,inner sep=0pt] (5+) {};
\draw (6,.75) node[draw,circle,fill=black,minimum size=5pt,inner sep=0pt] (6+) {};
\draw (7,.75) node[draw,circle,fill=black,minimum size=5pt,inner sep=0pt] (7+) {};
\draw (8,.75) node[draw,circle,fill=black,minimum size=5pt,inner sep=0pt] (8+) {};

\draw (1,0) node[draw,circle,fill=black,minimum size=5pt,inner sep=0pt] (1-) {};
\draw (2,0) node[draw,circle,fill=black,minimum size=5pt,inner sep=0pt] (2-) {};
\draw (3,0) node[draw,circle,fill=white,minimum size=5pt,inner sep=0pt] (3-) {};
\draw (4,0) node[draw,circle,fill=black,minimum size=5pt,inner sep=0pt] (4-) {};
\draw (5,0) node[draw,circle,fill=black,minimum size=5pt,inner sep=0pt] (5-) {};
\draw (6,0) node[draw,circle,fill=black,minimum size=5pt,inner sep=0pt] (6-) {};
\draw (7,0) node[draw,circle,fill=black,minimum size=5pt,inner sep=0pt] (7-) {};
\draw (8,0) node[draw,circle,fill=white,minimum size=5pt,inner sep=0pt] (8-) {};

\node at (8,1.1) {{$1$}};
\node at (7,1.1) {{$0$}};
\node at (6,1.1) {{$0$}};
\node at (7.1,-.3) {{$0$}};
\node at (6.1,-.3) {{$0$}};
\node at (4.9,-.3) {{$1$}};
\node at (3.9,-.3) {{$0$}};
\node at (2.1,-.3) {{$-1$}};
\node at (.9,-.3) {{$2$}};
\node at (1,1.1) {{$-2$}};
\node at (2,1.1) {{$1$}};
\node at (3,1.1) {{$2$}};

\draw (1-) to (2-);
\draw (4-) to (7-);
\draw (1+) to (3+);
\draw (6+) to (7+);
\draw [double distance=.8mm,postaction={decorate}] (8+) to (7+);

\draw [dashed] (1+) to [bend left=60] (3+);
\draw [dashed] (1-) to [bend right=60] (2-);
\draw [dashed] (4-) to [bend right=60] (7-);
\draw [dashed] (5-) to [bend right=60] (6-);

;\end{tikzpicture}\]
\caption{The simple eigenvalues of $\mf{p}_8^\C(
\{\alpha_8,\alpha_7,\alpha_6,\alpha_3, \alpha_2, \alpha_1
\dd \alpha_8,\alpha_7,\alpha_5,
\alpha_4,\alpha_3,\alpha_2\})$}
\label{fig:simple eigenvalues}
\end{figure}
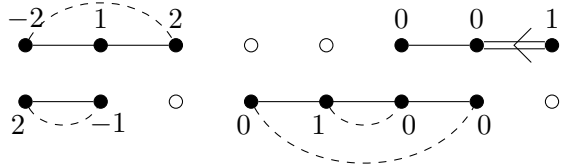

\end{example}

We have the following corollary of Lemma \ref{eigenvalue orbit} that will be used to prove symmetry and the unbroken property.  

\begin{theorem}\label{typeAsymmetricpositiveroot}
If $\sigma$ is a maximally connected component of type $A_k$, then $\sum_{i=1}^k \alpha_i (\widehat{F}) = 1$.  
\end{theorem}

\begin{proof}
If $k$ is odd, then $\alpha_i (\widehat{F}) = -\alpha_{k+1-i} (\widehat{F})$ for $i < \frac{k+1}{2}$, and $\alpha_{\frac{k+1}{2}} (\widehat{F}) = 1$.  
If $k$ is even, then $\alpha_i (\widehat{F}) = -\alpha_{k+1-i} (\widehat{F})$ for $i < \frac{k}{2}$, and $\alpha_{\frac{k}{2}} (\widehat{F}) + \alpha_{\frac{k}{2}+1} (\widehat{F}) = 1$.  
\end{proof}

If will be convenient to use the following notation.  
Let $\sigma$ be a maximally connected component of $\pi_1$, and let $\sigma=\{\alpha_k, \alpha_{k-1}, ..., \alpha_1\}$ be of type $A_k$.  The positive roots of $\sigma$ are of the form 
\[\alpha_j+\alpha_{j-1}+\dots +\alpha_i,\]
where $k\geq j\geq i\geq 1$.  
If $\alpha$ is a positive root with $j+i \neq k+1$, call $\overline{\alpha}$ its \textit{symmetric (positive) root}, where

\[\overline{\alpha}=
\begin{dcases}
\alpha_{i-1}+\alpha_{i-2} + \dots + \alpha_{k+1-j}, & \text{ if }k \geq j > i \geq \lc\frac{k}{2}\rc \geq 1; \\
\alpha_{i-1}+\alpha_{i-2} + \dots + \alpha_{k+1-j}, & \text{ if } \left|j - \lc\frac{k}{2}\rc\right| > \left|i - \lc\frac{k}{2}\rc\right|; \\
\alpha_{j+1}+ \alpha_{j+2} + \dots + \alpha_{k+1-i}, & \text{ if } \left|j - \lc\frac{k}{2}\rc\right| < \left|i - \lc\frac{k}{2}\rc\right|; \\
\alpha_{j+1}+ \alpha_{j+2} + \dots + \alpha_{k+1-i}, & \text{ if }k \geq \lc\frac{k}{2}\rc \geq j > i \geq 1. 
\end{dcases}\]

\noindent 
As a corollary of Theorem \ref{typeAsymmetricpositiveroot}, the symmetric roots $\alpha$ and $\overline{\alpha}$ satisfy the following relation:
\begin{eqnarray}\label{SymmPos}
\alpha(\widehat{F})+\overline{\alpha}(\widehat{F})=1
\end{eqnarray}
\noindent 
Symmetric roots satisfy the relation (\ref{SymmPos}), and since $\alpha(\widehat{F})=1$ when $j+i=k+1$, we say the positive root $\alpha$ has no symmetric root.  
We call $\overline{\alpha}(\widehat{F})$ the \textit{symmetric eigenvalue} of $\alpha(\widehat{F})$ or that $\overline{\alpha}(\widehat{F})$ is an eigenvalue \textit{symmetric to} $\alpha(\widehat{F})$.  

\begin{example}
Referring once again to Figure \ref{fig:simple eigenvalues}, consider the type-A maximally connected bottom component $\sigma = \{\alpha_5, \alpha_4, \alpha_3, \alpha_2\}$.  The root symmetric to $\alpha = \alpha_5+\alpha_4+\alpha_3$ is $\overline{\alpha} = \alpha_2$, and the root symmetric to $\alpha = \alpha_5+\alpha_4$ is $\overline{\alpha} = \alpha_3+\alpha_2$.  The root $\alpha = \alpha_5 + \alpha_4 + \alpha_3 + \alpha_2$ has no symmetric root.

\end{example}







\section{Proof of Theorem \ref{thm:main} }


\subsection{Symmetry}

We will partition the multiset of eigenvalues according to the
maximally connected components $\sigma$ of $\pi_1$ and $\pi_2$. 
We will prove that each member of this partition is 
symmetric and unbroken.

If $\sigma$ is a Type-A maximally connected
component of $\pi_1$, then we say the multiset of eigenvalues 
from $\sigma$ is
\begin{eqnarray}\label{S1}
\mathcal{E}(\sigma)=\{\beta(\widehat{F})\mid\beta\in\mathbb{N}\sigma\cap\Delta_+\}
\cup\{0^{\ceil{|\sigma|/2}}\},
\end{eqnarray}
and if $\sigma$ is not of Type A then
\begin{eqnarray}\label{S2}
\mathcal{E}(\sigma)=\{\beta(\widehat{F})\mid\beta\in\mathbb{N}\sigma\cap\Delta_+\}
\cup\{0^{|\sigma|}\},
\end{eqnarray}
\noindent 
where $\mathbb{N}\sigma = \left\{\sum c_i \alpha_i ~|~ c_i > 0 ~ \text{ and } \alpha_i \in \sigma\right\}$.  
\noindent
We proceed similarly if $\sigma$ is a maximally connected component of $\pi_2$ except that $\Delta_+$ is replaced by $\Delta_-$ in equations (\ref{S1}) and 
(\ref{S2}).


\begin{example}\label{ex:spectrum}
In the running example of Figure \ref{fig:simple eigenvalues}, on the top there is a single type-A component $\sigma_1=\{\alpha_8, \alpha_7, \alpha_6\}$ and a single type-C component $\sigma_2=\{\alpha_3, \alpha_2, \alpha_1\}$. There are two type-A components on the bottom:  $\sigma_3=\{\alpha_8, \alpha_7\}$ and 
$\sigma_4=\{\alpha_5, \alpha_4, \alpha_3, \alpha_2\}$.  To demonstrate the computation of $\mathcal{E}(\sigma_i)$, note,
for example, that the positive roots for the computation of $\mathcal{E}(\sigma_1)$ are elements of the set $B_{\sigma_1}=\{\alpha_8, \alpha_7, \alpha_6, \alpha_8+\alpha_7, \alpha_7+\alpha_6, \alpha_8+\alpha_7+\alpha_6  \}$.  Applying each of $\beta\in B_{\sigma_1}$ to $\widehat{F}$ yields the multiset $\{-2, 1, 2, -1, 3, 1\} = \{-2,-1,1,1,2,3  \}.$  Since, $|\sigma_1|=3$, we have by equation (\ref{S1}) that

$$
\mathcal{E}(\sigma_1) = \{-2,-1,1,1,2,3\}\cup\{0,0\} = \{-2,-1,0,0,1,1,2,3\}.
$$
The positive roots for the computation of 
$\mathcal{E}(\sigma_2)$ are elements of the set 

$$B_{\sigma_2}=\{\alpha_3, \alpha_2, \alpha_1, \alpha_3+\alpha_2, \alpha_2+\alpha_1, \alpha_3+\alpha_2+\alpha_1, 2\alpha_2+\alpha_1, \alpha_3 + 2\alpha_2+\alpha_1, 2\alpha_3+2\alpha_2+\alpha_1\},$$ 
\noindent
which gives

$$
\mathcal{E}(\sigma_2) =
\{0,0,0,1,1,1,1,1,1\}\cup \{0,0,0\} = 
\{0,0,0,0,0,0,1,1,1,1,1,1\}.
$$
\noindent Similar computations yield

$$\mathcal{E}(\sigma_3)= \{-1,0,1,2 \} ~\text{and}~
\mathcal{E}(\sigma_4)= \{0,0,0,0,0,0,1,1,1,1,1,1\}. $$
\noindent 
\end{example}

\begin{remark}
We will use the Greek letter $\sigma$ when referring to the simple roots of a maximally connected component of an orbit meander.  However, there are times it will be more convenient to consider the set of eigenvalues associated to a set of consecutive vertices in an orbit meander.  
If $A$ is a set of consecutive vertices in an orbit meander, let $\mathcal{E}(A)$ be the eigenvalues associated to $A$.  That is, if $A = \{v_k, v_{k-1}, ... ,v_j\}$, let $\sigma = \{\alpha_k, \alpha_{k-1}, ..., \alpha_j \}$.  We make the following notational convention:
\begin{eqnarray*}
\mathcal{E}(A):=\mathcal{E}(\sigma).  
\end{eqnarray*}
\end{remark}

\begin{lemma}
If $\mf{p}(\pi_1\dd \pi_2)$ is a Frobenius seaweed, then the
multisets of eigenvalues contributed by each maximally
connected component form a multiset partition of the spectrum
of the seaweed.
\end{lemma}

\begin{example}
We union the sets $\mathcal{E}(\sigma_i)$ from Example \ref{ex:spectrum} to get the spectrum.  This data is consolidated in the following table.  

\begin{table}[H]
\[\begin{tabular}{|l||l|l|l|l|l|l|}
\hline
Eigenvalue& -2 & -1 & 0 & 1 & 2 & 3 \\
\hline
Multiplicity & 1 & 2 & 15 & 15 & 2 & 1 \\
\hline
\end{tabular}\]
\caption{Spectrum with multiplicities}
\label{tab:eigenvalue}
\end{table}
\end{example}

Notice that in Example \ref{ex:spectrum}, $\mathcal{E}(\sigma_i)$ is symmetric about one half for each $i$.  
Furthermore, the multiplicities in Table \ref{tab:eigenvalue} form a symmetric distribution. 
As the following Theorem \ref{thm:symmetry} shows,  these observations are not coincidences. 

\begin{theorem}\label{thm:symmetry}
Let $\mf{p}(\pi_1\dd \pi_2)$ be a Frobenius seaweed.
For each maximally connected component $\sigma$ of $\pi_1$ or $\pi_2$,
let $r_i$ be the multiplicity of the eigenvalue $i$ in $\mathcal{E}(\sigma)$. 
The sequence $(i)$ is symmetric about one-half.  Moreover, $r_{-i} = r_{i+1}$ for each eigenvalue $i$.
\end{theorem}

\begin{proof}
We prove this in the case
that $\sigma$ is a maximally connected component of $\pi_1$.
The case that $\sigma$ is a maximally connected component of $\pi_2$
is similar and is omitted.

\setcounter{case}{0}
\begin{case}\label{symmetryA}
Type A
\end{case}
Suppose $\sigma$ is of type $A_k$.  
Then $\mathcal{E}(\sigma)$ is comprised of elements of the form 
\[\alpha_j(\widehat{F})+\alpha_{j-1}(\widehat{F})+\dots +\alpha_i(\widehat{F}),\]
where $k\geq j\geq i\geq 1$ along with $\lc \frac{k}{2}\rc$ zeros per Equation (\ref{S1}).  

Each element $\alpha(\widehat{F})$ in $\mathcal{E}(\sigma)$ with $j+i \neq k+1$ has a symmetric eigenvalue $\overline{\alpha}$ with $\alpha(\widehat{F})+\overline{\alpha}(\widehat{F}) = 1$.  
Moreover, there are $\lc \frac{k}{2}\rc$ positive roots 
$\alpha_j + \alpha_{j-1} ... + \alpha_i$ with $k \geq j \geq i \geq 1$ and $j+i = k+1$.  These satisfy
\[\alpha_j(\widehat{F})+\alpha_{j-1}(\widehat{F})+\dots +\alpha_i(\widehat{F})=1\]
 and are in bijective correspondence with the zeros from Equation (\ref{S1}).  Therefore, the multiset of eigenvalues from $\sigma$ is symmetric about one-half.  

\begin{case}\label{symmetryB}
Type B
\end{case}
Suppose $\sigma$ is of type B, with odd cardinality.
For convenience, reorder the indices of the simple roots
so that $\sigma=\{\alpha_{2k-1},\alpha_{2k-2},\dots ,\alpha_{1}\}$ as in 
Lemma \ref{eigenvalue table}.
Then $\mathcal{E}(\sigma)$ is comprised of elements of the form
\[\alpha_j(\widehat{F})+\alpha_{j-1}(\widehat{F})+\dots +\alpha_i(\widehat{F}),\]
where $2k-1\geq j\geq i\geq 1$, or
\[\alpha_{j}(\widehat{F})+\dots +\alpha_{i+1}(\widehat{F})+2\alpha_{i}(\widehat{F})+\dots +2\alpha_{1}(\widehat{F}),\]
where $2k-1\geq j>i\geq 2$, along with $2k-1$ zeros per Equation (\ref{S2}).

We need only consider eigenvalues $-1, 0, 1,$ and $2$, and these eigenvalues have multiplicities given by 

\[r_{-1}=\sum_{i=1}^{k-1}i=\binom{k}{2},\]
\[r_{0}=\left(\sum_{i=1}^{k-1}2i\right)+(2k-1)+\left(\sum_{i=1}^{k-2}i\right)
=\dfrac{k(3k-1)}{2},\]
\[r_{1}=\left(\sum_{i=1}^{k}i\right)+\left(\sum_{i=1}^{k-1}2i\right)
=\dfrac{k(3k-1)}{2},\]
\[r_{2}=\sum_{i=1}^{k-1}i=\binom{k}{2}.\]

If $\sigma$ is of type B with even cardinality,
we reorder the indices of the simple roots
so that $\sigma=\{\alpha_{2k},\dots ,\alpha_{1}\}$ as in 
Lemma \ref{eigenvalue table}. 
Then $\mathcal{E}(\sigma)$ is comprised of elements of the form
\[\alpha_j(\widehat{F})+\alpha_{j-1}(\widehat{F})+\dots +\alpha_i(\widehat{F}),\]
where $2k\geq j\geq i\geq 1$, or

\[\alpha_{j}(\widehat{F})+\dots +\alpha_{i+1}(\widehat{F})+2\alpha_{i}(\widehat{F})+\dots +2\alpha_{1}(\widehat{F}),\]
where $2k\geq j>i\geq 2$, along with $2k$ zeros per Equation (\ref{S2}).
Then

\[r_{-1}=\sum_{i=1}^{k-1}i=\binom{k}{2},\]
\[r_{0}=\left(\sum_{i=1}^{k}2(i-1)+1\right)+(2k)+\left(\sum_{i=1}^{k-1}i\right)
=3\binom{k+1}{2},\]
\[r_{1}=\left(\sum_{i=1}^{k}i+1\right)+(k)+\left(\sum_{i=1}^{k-1}2i\right)
=3\binom{k+1}{2},\]
\[r_{2}=\sum_{i=1}^{k-1}i=\binom{k}{2}.\]

\begin{case}\label{symmetryC}
Type C
\end{case}
Suppose $\sigma$ is of type C.  We again 
reorder the indices of the simple roots
so that $\sigma=\{\alpha_{k},\alpha_{k-1},\dots ,\alpha_{1}\}$ as in 
Lemma \ref{eigenvalue table}.
Then $\mathcal{E}(\sigma)$ is comprised of elements of the form
\[\alpha_j(\widehat{F})+\alpha_{j-1}(\widehat{F})+\dots +\alpha_i(\widehat{F}),\]
where $k\geq j\geq i\geq 1$, or
\[\alpha_j(\widehat{F})+\dots +\alpha_{i+1}(\widehat{F})+2\alpha_{i}(\widehat{F})+\dots +2\alpha_{2}(\widehat{F})
+\alpha_1(\widehat{F}),\]
where $k\geq j\geq i\geq 2$, along with $k$ zeros per Equation (\ref{S2}).

Counting the multiplicity of the eigenvalue 0, we see that 
there are $\binom{k-1}{2}$ eigenvalues of the form 
$\alpha_j(\widehat{F})+\dots +\alpha_i(\widehat{F})$ where $k\geq j>i\geq 2$, 
$k-1$ eigenvalues of the form $\alpha_i(\widehat{F})$ where $k\geq i\geq 2$,
and $k$ zeros per Equation (\ref{S2}) included in $\mathcal{E}(\sigma)$.
Thus $r_0=\binom{k-1}{2}+(k-1)+k=\binom{k+1}{2}$.

Counting the multiplicity of the eigenvalue 1, we see that 
there are $k$ eigenvalues of the form 
$\alpha_j(\widehat{F})+\dots +\alpha_1(\widehat{F})$ where $k\geq j\geq 1$, 
$k-1$ eigenvalues of the form 
$2\alpha_i(\widehat{F})+\dots 2\alpha_{2}(\widehat{F})+\alpha_1(\widehat{F})$ where $k\geq i\geq 2$, and $\binom{k-1}{2}$ eigenvalues of the form 
$\alpha_j(\widehat{F})+\dots +\alpha_{i+1}(\widehat{F})+2\alpha_{i}(\widehat{F})+\dots +2\alpha_{2}(\widehat{F})
+\alpha_1(\widehat{F})$ where $k\geq j>i\geq 2$.
It follows that $r_1=\binom{k+1}{2}$ as well.
\end{proof} 

We have the following corollary that will be used to prove the unbroken property.  

\begin{corollary}
Let $\mf{p}(\pi_1\dd \pi_2)$ be a Frobenius seaweed, and let $\sigma$ be a maximally connected component of type A, B, or C.  If $\mathcal{E}(\sigma)$ is an unbroken multiset, then $\mathcal{E}(\sigma) \cup [-\mathcal{E}(\sigma)]$ is an unbroken multiset.  Moreover, if $x$ is symmetric to any eigenvalue in $\mathcal{E}(\sigma)$, then $\mathcal{E}(\sigma) \cup \{x\}$ is an unbroken multiset.  
\end{corollary}

\subsection{Unbroken}

A key concept for showing the spectrum is unbroken is a ``U-turn" in an orbit of the orbit meander.  
By U-turn, we mean an application of Lemma \ref{eigenvalue orbit} 
in the case that $(\alpha_i,i_1\alpha_i)<0$.  Note that U-turns can occur only in type $A_{2k}$.  
For example, the orbit meander in Figure \ref{fig:simple eigenvalues} has two U-turns: one in the orbit 
$\{\alpha_7,\alpha_6,\alpha_8\}$ and one in the orbit $\{\alpha_4,\alpha_3\}$.  We will find it convenient to break U-turns into two types of U-turns: \textit{right U-turns} and \textit{left U-turns}.  

Arrange all orbits so that the first entry is a fixed point in $\pi_1\cap\pi_2$.
If a U-turn involves a dashed edge
from $v_i^-$ to $v_{i-1}^-$, or a dashed edge from $v_i^+$ to $v_{i+1}^+$, we call this a right U-turn.  
Similarly, if a U-turn involves a dashed edge
from $v_i^-$ to $v_{i+1}^-$, or a dashed edge from $v_i^+$ to $v_{i-1}^+$, we call this a left U-turn.  See Figure \ref{fig:uturns}; the right U-turn is represented by a red dashed arc, and the left U-turns are blue.

\begin{figure}[H]
\[\begin{tikzpicture}
[decoration={markings,mark=at position 0.6 with 
{\arrow{angle 90}{>}}}]

\draw (1,.75) node[draw,circle,fill=black,minimum size=5pt,inner sep=0pt] (1+) {};
\draw (2,.75) node[draw,circle,fill=black,minimum size=5pt,inner sep=0pt] (2+) {};
\draw (3,.75) node[draw,circle,fill=black,minimum size=5pt,inner sep=0pt] (3+) {};
\draw (4,.75) node[draw,circle,fill=black,minimum size=5pt,inner sep=0pt] (4+) {};
\draw (5,.75) node[draw,circle,fill=black,minimum size=5pt,inner sep=0pt] (5+) {};
\draw (6,.75) node[draw,circle,fill=black,minimum size=5pt,inner sep=0pt] (6+) {};
\draw (7,.75) node[draw,circle,fill=white,minimum size=5pt,inner sep=0pt] (7+) {};
\draw (8,.75) node[draw,circle,fill=black,minimum size=5pt,inner sep=0pt] (8+) {};
\draw (9,.75) node[draw,circle,fill=black,minimum size=5pt,inner sep=0pt] (9+) {};
\draw (10,.75) node[draw,circle,fill=black,minimum size=5pt,inner sep=0pt] (10+) {};

\draw (1,0) node[draw,circle,fill=black,minimum size=5pt,inner sep=0pt] (1-) {};
\draw (2,0) node[draw,circle,fill=black,minimum size=5pt,inner sep=0pt] (2-) {};
\draw (3,0) node[draw,circle,fill=black,minimum size=5pt,inner sep=0pt] (3-) {};
\draw (4,0) node[draw,circle,fill=black,minimum size=5pt,inner sep=0pt] (4-) {};
\draw (5,0) node[draw,circle,fill=white,minimum size=5pt,inner sep=0pt] (5-) {};
\draw (6,0) node[draw,circle,fill=black,minimum size=5pt,inner sep=0pt] (6-) {};
\draw (7,0) node[draw,circle,fill=black,minimum size=5pt,inner sep=0pt] (7-) {};
\draw (8,0) node[draw,circle,fill=black,minimum size=5pt,inner sep=0pt] (8-) {};
\draw (9,0) node[draw,circle,fill=black,minimum size=5pt,inner sep=0pt] (9-) {};
\draw (10,0) node[draw,circle,fill=white,minimum size=5pt,inner sep=0pt] (10-) {};

\draw (1-) to (4-);
\draw (6-) to (9-);
\draw (1+) to (6+);
\draw (8+) to (9+);
\draw [double distance=.8mm,postaction={decorate}] (10+) to (9+);

\draw [dashed] (1+) to [bend left=60] (6+);
\draw [dashed] (2+) to [bend left=60] (5+);
\draw [dashed, color=red, line width = 1.2pt] (3+) to [bend left=60] (4+);
\draw [dashed] (1-) to [bend right=60] (4-);
\draw [dashed, color=blue, line width = 1.2pt] (2-) to [bend right=60] (3-);
\draw [dashed] (6-) to [bend right=60] (9-);
\draw [dashed, color=blue, line width = 1.2pt] (7-) to [bend right=60] (8-);

;\end{tikzpicture}\]
\caption{The orbit meander of $\mf{p}_{10}^\C(
\{\alpha_{10}, \alpha_9,\alpha_8,\alpha_7, \alpha_6,\alpha_5,\alpha_3,\alpha_2, \alpha_1 
\dd 
\alpha_{10}, \alpha_9,\alpha_8,\alpha_7,\alpha_5, \alpha_4,\alpha_3, \alpha_2\})$ with U-turns highlighted}
\label{fig:uturns}
\end{figure}
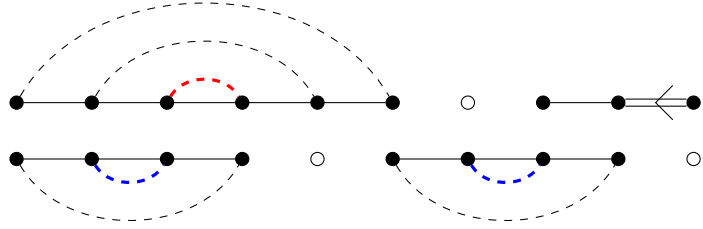

The $<i_1i_2>$ orbits of the seaweed in Figure \ref{fig:uturns} are $\{\alpha_3, \alpha_4\}$, 
$\{\alpha_2, \alpha_5, \alpha_{10}, \alpha_7, \alpha_8, \alpha_9, \alpha_6\}$, and $\{\alpha_1\}$.  
The orbit $\{\alpha_3, \alpha_4\}$ has a left U-turn, and the orbit $\{\alpha_2, \alpha_5, \alpha_{10}, \alpha_7, \alpha_8, \alpha_9, \alpha_6\}$ has a right U-turn and a left U-turn.  
As it turns out, an orbit cannot have more than two U-turns.  
Moreover, if an orbit has two U-turns, one must be a right U-turn and one must be a left U-turn.  
Combining this observation and Lemma \ref{eigenvalue orbit}, we see that the simple eigenvalues for a Frobenius seaweed in types A, B, or C are bounded in absolute value by $3$. 
We record this in the following Lemma.



\begin{lemma}\label{lem:simple}
The absolute value of every simple eigenvalue is either 0, 1, 2, or 3.
\end{lemma}

\begin{proof}

The orbit meander for a Frobenius type-A seaweed has exactly 2 maximally connected
components of even size. It follows that the largest absolute value
of a simple eigenvalue is 3.

For seaweeds of types B and C, it's possible to 
have more than 2 maximally connected components of even size.
Indeed, start from a fixed point contained in $\pi_1\cap\pi_2$.
If the orbit never U-turns, then we're done. Otherwise follow the
orbit until its first U-turn. 
Without loss of generality, we proceed with the proof in the case that the first U-turn is a right U-turn.

The orbit may make a second U-turn. However, it cannot make a second
right U-turn, as the orbit would self-intersect 
and never terminate at an element of $\pi_\cup$. Therefore, a second U-turn must be to the left. But now the orbit has
previously traced edges on both the left and right. Additional U-turns
are not allowed.  The result follows.
\end{proof}

\begin{remark}
Using the notation above, it follows from the proof of Lemma \ref{lem:simple} that the following table gives the possible simple eigenvalues for components of the various types.  

\begin{table}[H]
\[\begin{tabular}{|c|c|}
\hline
Component type & Possible values of $\alpha_i(\widehat{F})$  \\
\hline
\hline
A & 
$-2, -1, 0, 1, 2, 3$ \\
\hline
B & $-1, 0, 1$  \\
\hline
C & $0,1$\\
\hline
\end{tabular}\]
\caption{Values of $\alpha_i(\widehat{F})$}
\label{tab:ABCsimple eigenvalue}
\end{table}

\end{remark}

To prove the unbroken property for types B and C seaweeds, we will find it convenient to express these seaweeds in terms of sequences of flags defining them.   Let $C_{\leq n}$ denote the set of sequences of positive integers whose sum 
is less than or equal to $n$, and call each integer in the string a \textit{part}.
Let $\mathcal{P}(X)$ denote the power set of a set $X$.
Given $\ul{a}=(a_1,a_2,\dots ,a_m)\in \Cn$, define a bijection
$\varphi:\Cn\rightarrow \mathcal{P}(\Pi)$ by
\[\varphi(\ul{a})=\{\alpha_{n+1-a_1},\alpha_{n+1-(a_1+a_2)},
\dots ,\alpha_{n+1-(a_1+a_2+\dots +a_m)}\}.\]
Then define \[\mf{p}_n(\ul{a} \dd \ul{b})=
\mf{p}\left(\Pi\setminus\varphi(\ul{a}) \dd \Pi\setminus\varphi(\ul{b})\right),\]
and let $\mathcal{M}_n(\ul{a} \dd \ul{b})$ denote the orbit meander of 
$\mf{p}_n(\ul{a} \dd \ul{b})$.

\begin{example}
For example,
$\mf{p}_8^\C(\{\alpha_6,\alpha_5,\alpha_4,\alpha_3
\dd \alpha_8,\alpha_7,\alpha_6,\alpha_4,
\alpha_3,\alpha_2,\alpha_1\})=
\mf{p}^\C_8((1,1,5,1)\dd (4))$.
\end{example}

If $a_1+\dots +a_m=n$, then each part $a_i$ corresponds to a 
maximally connected component $\sigma$ of cardinality $|\sigma|=a_i-1$,
all of which are of type A.
If $a_1+\dots +a_m=r<n$, then each part $a_i$ corresponds to a 
type-A maximally connected component $\sigma$ of cardinality $|\sigma|=a_i-1$,
and there is one additional maximally connected component of cardinality $n-r$,
which is of type B or C if $n-r>1$.


The following ``Winding-up" lemma can be used to develop any Frobenius orbit meander of any size or configuration.   It can be regarded as the inverse graph-theoretic rendering of Panyushev's well-known reduction \textbf{\cite{Panyushev1}}.


\begin{lemma}[Coll et al. \textbf{\cite{Meanders3}}, Lemma 4.2]\label{Expansion}
Let $\mathcal{M}_n(\ul{c} \dd \ul{d})$ be any type-B or type-C Frobenius
orbit meander, and without loss
of generality, assume that $\sum c_i=q+\sum d_i=n$.
Then $\mathcal{M}_n(\ul{c} \dd \ul{d})$ is the result of a sequence of the 
following moves starting from $\mathcal{M}_q(1^q \dd \emptyset)$.
Starting with an orbit meander 
$\mathcal{M}=\mathcal{M}_n\left(a_{1},a_{2}, \dots ,a_{m}\dd b_{1},b_{2}, \dots ,b_{t}\right)$, 
create an orbit meander $\mathcal{M}'$ by one of of the following:

\begin{enumerate}
\item {\bf Block Creation:}
$\displaystyle
\mathcal{M}'=\mathcal{M}_{n+a_1}(2a_{1},a_{2}, \dots ,a_{m}\dd a_1, b_{1},b_{2}, \dots ,b_{t})$,
\item {\bf Rotation Expansion:}
$\displaystyle
\mathcal{M}'=\mathcal{M}_{n+a_1-b_1}(2a_{1}-b_1,a_{2},a_3, \dots ,a_{m}\dd a_1, b_{2},b_{3}, \dots ,b_{t})$, provided that $a_1>b_1$,
\item {\bf Pure Expansion:}
$\displaystyle
\mathcal{M}'=\mathcal{M}_{n+a_2}(a_{1}+2a_2,a_{3},a_4, \dots ,a_{m}\dd a_2,b_{1},b_{2}, \dots ,b_{t})$,
\item {\bf Flip-Up:}
$\displaystyle
\mathcal{M}'=\mathcal{M}_{n}(b_{1},b_{2}, \dots ,b_{t}\dd a_{1},a_{2}, \dots ,a_{m})$.
\end{enumerate}


\end{lemma}

\begin{remark}
While we previously found it convenient to order the vertices of an orbit meander from right to left, to ease notation in the following proof, we will find it convenient to relabel the vertices going from left to right.  That is, if an orbit meander is defined by vertices
$\{v_n^+, v_{n-1}^+, ..., v_1^+\}$ and 
$\{v_n^-, v_{n-1}^-, ..., v_1^-\}$, redefine the orbit meander by 

$$\{v_n^+, v_{n-1}^+, ..., v_1^+\} \mapsto \{w_1^+, w_{2}^+, ..., w_n^+\}$$
$$\{v_n^-, v_{n-1}^-, ..., v_1^-\} \mapsto \{w_1^-, w_{2}^-, ..., w_n^-\}$$
\noindent 
This is done so that the induction is done on the block whose leftmost vertex is labeled $w_1^+$ rather than $v_n^+$.  
\end{remark}

\begin{theorem}\label{thm:unbroken}
If $\mathcal{M}(\pi_1\dd\pi_2)=\mathcal{M}(\ul{a}\dd\ul{b})$ is any Frobenius orbit meander, 
then
$\mathcal{E}(\sigma)$ is unbroken for every maximally connected component $\sigma$. 
\end{theorem}

\begin{proof} 
The proof is by  
induction on the number of Winding-up moves from Lemma \ref{Expansion} and
that $\mathcal{E}(\sigma)$ is unbroken for every maximally connected component $\sigma$.
Since $\mathcal{E}(\sigma)$ always contains $0$, Theorem \ref{thm:unbroken} implies the unbroken property of
Theorem \ref{thm:main}.

The base of the induction is an orbit meander $\mathcal{M}_q(1^q\dd \emptyset)$
for either type B or C where $q$ is a positive integer.
There is one maximally connected component $\sigma$, of type B or type C, and by Theorem \ref{thm:symmetry}, $\mathcal{E}(\sigma)$ is unbroken either way.  

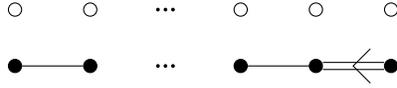
\begin{figure}[H]
\[\begin{tikzpicture}
[decoration={markings,mark=at position 0.6 with 
{\arrow{angle 90}{>}}}]

\draw (3,-.75) node[draw,circle,fill=black,minimum size=5pt,inner sep=0pt] (3+) {};
\draw (4,-.75) node[draw,circle,fill=black,minimum size=5pt,inner sep=0pt] (4+) {};
\draw (4.9,-.75) node[draw,circle,fill=black,minimum size=1pt,inner sep=0pt] (4.9+) {};
\draw (5,-.75) node[draw,circle,fill=black,minimum size=1pt,inner sep=0pt] (5+) {};
\draw (5.1,-.75) node[draw,circle,fill=black,minimum size=1pt,inner sep=0pt] (5.1+) {};
\draw (6,-.75) node[draw,circle,fill=black,minimum size=5pt,inner sep=0pt] (6+) {};
\draw (7,-.75) node[draw,circle,fill=black,minimum size=5pt,inner sep=0pt] (7+) {};
\draw (8,-.75) node[draw,circle,fill=black,minimum size=5pt,inner sep=0pt] (8+) {};

\draw (3,0) node[draw,circle,fill=white,minimum size=5pt,inner sep=0pt] (3-) {};
\draw (4,0) node[draw,circle,fill=white,minimum size=5pt,inner sep=0pt] (4-) {};
\draw (4.9,0) node[draw,circle,fill=black,minimum size=1pt,inner sep=0pt] (4.9-) {};
\draw (5,0) node[draw,circle,fill=black,minimum size=1pt,inner sep=0pt] (5-) {};
\draw (5.1,0) node[draw,circle,fill=black,minimum size=1pt,inner sep=0pt] (5.1-) {};
\draw (6,0) node[draw,circle,fill=white,minimum size=5pt,inner sep=0pt] (6-) {};
\draw (7,0) node[draw,circle,fill=white,minimum size=5pt,inner sep=0pt] (7-) {};
\draw (8,0) node[draw,circle,fill=white,minimum size=5pt,inner sep=0pt] (8-) {};

\draw (3+) to (4+);
\draw (6+) to (7+);
\draw [double distance=.8mm,postaction={decorate}] (8+) to (7+);

;\end{tikzpicture}\]
\caption{The Frobenius orbit meander $\mathcal{M}_q^\C(1^q\dd \emptyset)$}
\label{typeCbase}
\end{figure}

Let $\mathcal{M} = \mathcal{M}(\ul{a}\dd\ul{b})$ be a type-B or type-C orbit meander.  Suppose $\mathcal{E}(\sigma)$ is unbroken for each maximally connected component $\sigma$ in $\mathcal{M}$.  Let $\mathcal{M}'$ be the orbit meander resulting from applying one of the Winding-up moves from Lemma \ref{Expansion} to $\mathcal{M}$.  
For the inductive step, we need not consider the Flip-up move since this merely replaces $\mathcal{M}$ with an inverted isomorphic copy and consequently has no effect on the eigenvalue calculations. 
So, there are three cases we need to consider: block creation, rotation expansion, and pure expansion.  
We show $\mathcal{E}(\sigma)$ is unbroken for each maximally connected component $\sigma$ in $\mathcal{M}'$ that is not in $\mathcal{M}$.  
The following sets of vertices will assist in the computation of the eigenvalues.  
Note that the ordering of the sets in columns three and four of the following table are set up so that the vertices in the sets are in the order of the inducted-upon orbit meander.  

\newpage  

\begin{table}[H]
\begin{center}
\scalebox{0.83}{
\begin{tabular}{|c|c|c|c|}
\hline
\textbf{Move} & \textbf{Base} $\mathcal{M}$ & \textbf{New} $\mathcal{M}'$ & \textbf{New $\sigma$'s} \\
\hline
\hline
$\begin{array}{cc}
     \textbf{Block}  \\
     \textbf{Creation} 
\end{array}$ 
& 
$\begin{array}{l}
     A = \{w_1^+, w_{2}^+, ..., w_{a_1-1}^+\}  
\end{array}$ 
&
$\begin{array}{llll}
     B' = \{w_1^+, w_2^+, ..., w_{a_1-1}^+\}  \\
     C' = \{w_{a_1}^+\}  \\
     A' = \{w_{a_1+1}^+, w_{a_1+2}^+, ..., w_{2a_1-1}^+\} \\
     D' = \{w_1^-, w_2^-, ..., w_{a_1-1}^-\}  
\end{array}$  
&
$\begin{array}{ll}
     \sigma_1 = B' \cup C' \cup A'  \\ 
     \sigma_2 = D'  
\end{array}$ \\ 
\hline
$\begin{array}{cc}
     \textbf{Rotation}  \\
     \textbf{Expansion} 
\end{array}$ 
& 
$\begin{array}{lll}
     A = \{w_1^+, w_2^+, ..., w_{b_1-1}^+\}  \\ 
     B = \{w_{b_1}^+, w_{b_1+1}^+, ..., w_{a_1-1}^+\} \\
     C = \{w_1^-, w_2^-, ..., w_{b_1-1}^-\} 
\end{array}$
&
$\begin{array}{llll}
     A' = \{w_1^+, w_2^+, ..., w_{a_1-b_1}^+\}  \\
     C' = \{w_{a_1-b_1+1}^+, w_{a_1-b_1+2}^+, ..., w_{a_1-1}^+\}  \\
     B' = \{w_{a_1}^+, w_{a_1+1}^+, ..., w_{2a_1-b_1-1}^+\} \\
     D'= \{w_1^-, w_2^-, ...,w_{a_1-1}^-\} 
\end{array}$  
&
$\begin{array}{ll}
     \sigma_1 = A' \cup C' \cup B'  \\
     \sigma_2 = D'
\end{array}$ \\
\hline
$\begin{array}{cc}
     \textbf{Pure}  \\
     \textbf{Expansion} 
\end{array}$ & 
$\begin{array}{ll}
     A = \{w_1^+, w_2^+, ..., w_{a_1-1}^+\}  \\ 
     B = \{w_{a_1+1}^+, w_{a_1+2}^+, ..., w_{a_1+a_2-1}^+\}
\end{array}$
&
$\begin{array}{lll}
     B' = \{w_1^+, w_2^+, ..., w_{a_2-1}^+\} \\
     E' = \{w_{a_2}^+\} \\
     A' = \{w_{a_2+1}^+, w_{a_2+2}^+, ..., w_{a_1+a_2-1}\}  \\
     F' = \{w_{a_1+a_2}^+\} \\
     C' = \{w_{a_1+a_2+1}^+, w_{a_1+a_2+2}^+, ..., w_{a_1+2a_2-1}^+\} \\
     D' = \{w_1^-, w_2^-, ..., w_{a_2-1}^-\}
\end{array}$  
&
$\begin{array}{ll}
     \sigma_1 = B' \cup E' \cup A' \cup F' \cup C'  \\
     \sigma_2 = D' 
\end{array}$ \\
\hline
\end{tabular}
}
\caption{Sets in the Induction Proof}
\label{tab:induction}
\end{center}
\end{table}


Gray vertices in the orbit meanders associated to each Winding up move are vertices not impacted by the induction.  Furthermore, for brevity, the orbit meanders do not extend beyond the relevant top blocks.  

\setcounter{case}{0}
\begin{case}\label{unbroken1}
Block Creation:
\end{case}

\begin{figure}[H]
\[\begin{tikzpicture}
[decoration={markings,mark=at position 0.6 with 
{\arrow{angle 90}{>}}}]

\draw (.5,1) -- (4.5,1);
\draw (.5,.25) -- (4.5,.25);
\draw (.5,.25) -- (.5,1);
\draw (4.5,.25) -- (4.5,1);

\draw (1,.75) node[draw,circle,fill=black,minimum size=5pt,inner sep=0pt] (1+) {};
\draw (2,.75) node[draw,circle,fill=black,minimum size=5pt,inner sep=0pt] (2+) {};
\draw (3,.75) node[draw,circle,fill=black,minimum size=5pt,inner sep=0pt] (3+) {};
\draw (4,.75) node[draw,circle,fill=black,minimum size=5pt,inner sep=0pt] (4+) {};
\draw (5,.75) node[draw,circle,fill=white,minimum size=5pt,inner sep=0pt] (5+) {};

\draw (1,0) node[draw,circle,fill=gray,minimum size=5pt,inner sep=0pt] (1-) {};
\draw (2,0) node[draw,circle,fill=gray,minimum size=5pt,inner sep=0pt] (2-) {};
\draw (3,0) node[draw,circle,fill=gray,minimum size=5pt,inner sep=0pt] (3-) {};
\draw (4,0) node[draw,circle,fill=gray,minimum size=5pt,inner sep=0pt] (4-) {};
\draw (5,0) node[draw,circle,fill=gray,minimum size=5pt,inner sep=0pt] (5-) {};

\draw (1+) to (2+);
\draw (2+) to (3+);
\draw (3+) to (4+);

\draw [dashed] (1+) to [bend left=60] (4+);
\draw [dashed] (2+) to [bend left=60] (3+);

\node at (.5,1.25) {$A$};
\node at (1,.5) {$w_1^+$};
\node at (2,.5) {$w_2^+$};
\node at (3,.5) {$w_3^+$};
\node at (4,.5) {$w_4^+$};

;\end{tikzpicture}
\hspace{1cm}
\begin{tikzpicture}
[decoration={markings,mark=at position 0.6 with 
{\arrow{angle 90}{>}}}]

\draw (.5,1.5) -- (4.5,1.5);
\draw (.5,.75) -- (4.5,.75);
\draw (.5,.75) -- (.5,1.5);
\draw (4.5,.75) -- (4.5,1.5);

\draw (5.5,1.5) -- (9.5,1.5);
\draw (5.5,.75) -- (9.5,.75);
\draw (5.5,.75) -- (5.5,1.5);
\draw (9.5,.75) -- (9.5,1.5);

\draw (4.6,1.5) -- (5.4,1.5);
\draw (4.6,.75) -- (5.4,.75);
\draw (4.6,.75) -- (4.6,1.5);
\draw (5.4,.75) -- (5.4,1.5);

\draw (.5,-.25) -- (4.5,-.25);
\draw (.5,.6) -- (4.5,.6);
\draw (.5,.6) -- (.5,-.25);
\draw (4.5,.6) -- (4.5,-.25);

\draw (1,1.25) node[draw,circle,fill=black,minimum size=5pt,inner sep=0pt] (1+) {};
\draw (2,1.25) node[draw,circle,fill=black,minimum size=5pt,inner sep=0pt] (2+) {};
\draw (3,1.25) node[draw,circle,fill=black,minimum size=5pt,inner sep=0pt] (3+) {};
\draw (4,1.25) node[draw,circle,fill=black,minimum size=5pt,inner sep=0pt] (4+) {};
\draw (5,1.25) node[draw,circle,fill=black,minimum size=5pt,inner sep=0pt] (5+) {};
\draw (6,1.25) node[draw,circle,fill=black,minimum size=5pt,inner sep=0pt] (6+) {};
\draw (7,1.25) node[draw,circle,fill=black,minimum size=5pt,inner sep=0pt] (7+) {};
\draw (8,1.25) node[draw,circle,fill=black,minimum size=5pt,inner sep=0pt] (8+) {};
\draw (9,1.25) node[draw,circle,fill=black,minimum size=5pt,inner sep=0pt] (9+) {};
\draw (10,1.25) node[draw,circle,fill=white,minimum size=5pt,inner sep=0pt] (10+) {};

\draw (1,0) node[draw,circle,fill=black,minimum size=5pt,inner sep=0pt] (1-) {};
\draw (2,0) node[draw,circle,fill=black,minimum size=5pt,inner sep=0pt] (2-) {};
\draw (3,0) node[draw,circle,fill=black,minimum size=5pt,inner sep=0pt] (3-) {};
\draw (4,0) node[draw,circle,fill=black,minimum size=5pt,inner sep=0pt] (4-) {};
\draw (5,0) node[draw,circle,fill=white,minimum size=5pt,inner sep=0pt] (5-) {};
\draw (6,0) node[draw,circle,fill=gray,minimum size=5pt,inner sep=0pt] (6-) {};
\draw (7,0) node[draw,circle,fill=gray,minimum size=5pt,inner sep=0pt] (7-) {};
\draw (8,0) node[draw,circle,fill=gray,minimum size=5pt,inner sep=0pt] (8-) {};
\draw (9,0) node[draw,circle,fill=gray,minimum size=5pt,inner sep=0pt] (9-) {};
\draw (10,0) node[draw,circle,fill=gray,minimum size=5pt,inner sep=0pt] (10-) {};

\draw (1+) to (2+);
\draw (2+) to (3+);
\draw (3+) to (4+);
\draw (4+) to (5+);
\draw (5+) to (6+);
\draw (6+) to (7+);
\draw (7+) to (8+);
\draw (8+) to (9+);
\draw (1-) to (2-);
\draw (2-) to (3-);
\draw (3-) to (4-);

\draw [dashed] (1+) to [bend left=60] (9+);
\draw [dashed] (2+) to [bend left=60] (8+);
\draw [dashed] (3+) to [bend left=60] (7+);
\draw [dashed] (4+) to [bend left=90] (6+);
\draw [dashed] (1-) to [bend right=60] (4-);
\draw [dashed] (2-) to [bend right=60] (3-);

\node at (.5,-.5) {$D'$};
\node at (.5,1.75) {$B'$};
\node at (5,1.75) {$C'$};
\node at (9.5,1.75) {$A'$};
\node at (1,1) {$w_1^+$};
\node at (2,1) {$w_2^+$};
\node at (3,1) {$w_3^+$};
\node at (4,1) {$w_4^+$};
\node at (5,1) {$w_5^+$};
\node at (6,1) {$w_6^+$};
\node at (7,1) {$w_7^+$};
\node at (8,1) {$w_8^+$};
\node at (9,1) {$w_9^+$};
\node at (1,.35) {$w_1^-$};
\node at (2,.35) {$w_2^-$};
\node at (3,.35) {$w_3^-$};
\node at (4,.35) {$w_4^-$};

;\end{tikzpicture}
\]
\caption{Block Creation applied to $\mathcal{M}$ with $a_1 = 5$ (left) to obtain $\mathcal{M}'$ (right)}
\label{fig:InductionBlockCreation}
\end{figure}
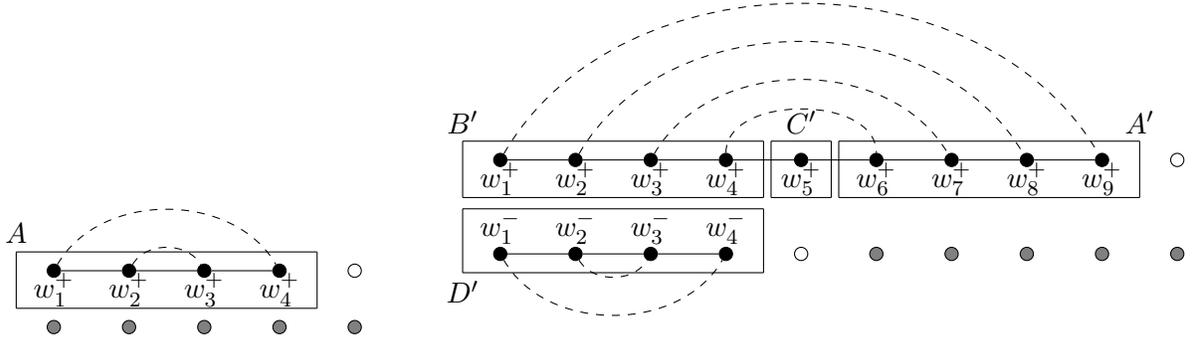

We have  

\noindent 
\begin{eqnarray}
\label{ind1.1} \mathcal{E}(A') &=& \mathcal{E}(A),\\
\label{ind1.2} \mathcal{E}(B') &=& -\mathcal{E}(A), \\
\label{ind1.3} \mathcal{E}(D') &=& \mathcal{E}(A).
\end{eqnarray} 
Note that $\mathcal{E}(A)$ is unbroken by induction.  So, by Equation (\ref{ind1.3}), $\mathcal{E}(\sigma_2)$ is unbroken.  
Any number in 
$\mathcal{E}(A' \cup C')$
is in either $\mathcal{E}(A')$ or $\mathcal{E}(A')+1$, so by equation (\ref{ind1.1}), $\mathcal{E}(A' \cup C')$ is unbroken.  By symmetry and equation (\ref{ind1.2}), 
$\mathcal{E}(B' \cup C')$ is unbroken.  Hence
$\mathcal{E}(B' \cup C') \cup \mathcal{E}(A' \cup C')$ is unbroken.  
Any remaining eigenvalue in $\mathcal{E}(\sigma_1)$ is either one or symmetric to an eigenvalue in $\mathcal{E}(B' \cup C') \cup \mathcal{E}(A' \cup C')$,  
so $\mathcal{E}(\sigma_1)$ is unbroken.  


\begin{case}\label{unbroken2}
Rotation Expansion:
\end{case}

\begin{figure}[H]
\[\begin{tikzpicture}
[decoration={markings,mark=at position 0.6 with 
{\arrow{angle 90}{>}}}]

\draw (.5,1.5) -- (2.5,1.5);
\draw (.5,.75) -- (2.5,.75);
\draw (.5,.75) -- (.5,1.5);
\draw (2.5,.75) -- (2.5,1.5);

\draw (2.6,1.5) -- (4.5,1.5);
\draw (2.6,.75) -- (4.5,.75);
\draw (2.6,.75) -- (2.6,1.5);
\draw (4.5,.75) -- (4.5,1.5);

\draw (.5,-.25) -- (2.5,-.25);
\draw (.5,.6) -- (2.5,.6);
\draw (.5,.6) -- (.5,-.25);
\draw (2.5,.6) -- (2.5,-.25);

\draw (1,1.25) node[draw,circle,fill=black,minimum size=5pt,inner sep=0pt] (1+) {};
\draw (2,1.25) node[draw,circle,fill=black,minimum size=5pt,inner sep=0pt] (2+) {};
\draw (3,1.25) node[draw,circle,fill=black,minimum size=5pt,inner sep=0pt] (3+) {};
\draw (4,1.25) node[draw,circle,fill=black,minimum size=5pt,inner sep=0pt] (4+) {};
\draw (5,1.25) node[draw,circle,fill=white,minimum size=5pt,inner sep=0pt] (5+) {};

\draw (1,0) node[draw,circle,fill=black,minimum size=5pt,inner sep=0pt] (1-) {};
\draw (2,0) node[draw,circle,fill=black,minimum size=5pt,inner sep=0pt] (2-) {};
\draw (3,0) node[draw,circle,fill=white,minimum size=5pt,inner sep=0pt] (3-) {};
\draw (4,0) node[draw,circle,fill=gray,minimum size=5pt,inner sep=0pt] (4-) {};
\draw (5,0) node[draw,circle,fill=gray,minimum size=5pt,inner sep=0pt] (5-) {};

\draw (1+) to (2+);
\draw (2+) to (3+);
\draw (3+) to (4+);

\draw [dashed] (1+) to [bend left=60] (4+);
\draw [dashed] (2+) to [bend left=60] (3+);
\draw [dashed] (1-) to [bend right=60] (2-);

\node at (.5,1.75) {$A$};
\node at (4.5,1.75) {$B$};
\node at (.5,-.5) {$C$};
\node at (1,1) {$w_1^+$};
\node at (2,1) {$w_2^+$};
\node at (3,1) {$w_3^+$};
\node at (4,1) {$w_4^+$};
\node at (1,.35) {$w_1^-$};
\node at (2,.35) {$w_2^-$};

;\end{tikzpicture}
\hspace{1cm}
\begin{tikzpicture}
[decoration={markings,mark=at position 0.6 with 
{\arrow{angle 90}{>}}}]

\draw (.5,1.5) -- (2.5,1.5);
\draw (.5,.75) -- (2.5,.75);
\draw (.5,.75) -- (.5,1.5);
\draw (2.5,.75) -- (2.5,1.5);

\draw (4.5,1.5) -- (6.5,1.5);
\draw (4.5,.75) -- (6.5,.75);
\draw (4.5,.75) -- (4.5,1.5);
\draw (6.5,.75) -- (6.5,1.5);

\draw (2.6,1.5) -- (4.4,1.5);
\draw (2.6,.75) -- (4.4,.75);
\draw (2.6,.75) -- (2.6,1.5);
\draw (4.4,.75) -- (4.4,1.5);

\draw (.5,-.25) -- (4.5,-.25);
\draw (.5,.6) -- (4.5,.6);
\draw (.5,.6) -- (.5,-.25);
\draw (4.5,.6) -- (4.5,-.25);

\draw (1,1.25) node[draw,circle,fill=black,minimum size=5pt,inner sep=0pt] (1+) {};
\draw (2,1.25) node[draw,circle,fill=black,minimum size=5pt,inner sep=0pt] (2+) {};
\draw (3,1.25) node[draw,circle,fill=black,minimum size=5pt,inner sep=0pt] (3+) {};
\draw (4,1.25) node[draw,circle,fill=black,minimum size=5pt,inner sep=0pt] (4+) {};
\draw (5,1.25) node[draw,circle,fill=black,minimum size=5pt,inner sep=0pt] (5+) {};
\draw (6,1.25) node[draw,circle,fill=black,minimum size=5pt,inner sep=0pt] (6+) {};
\draw (7,1.25) node[draw,circle,fill=white,minimum size=5pt,inner sep=0pt] (7+) {};

\draw (1,0) node[draw,circle,fill=black,minimum size=5pt,inner sep=0pt] (1-) {};
\draw (2,0) node[draw,circle,fill=black,minimum size=5pt,inner sep=0pt] (2-) {};
\draw (3,0) node[draw,circle,fill=black,minimum size=5pt,inner sep=0pt] (3-) {};
\draw (4,0) node[draw,circle,fill=black,minimum size=5pt,inner sep=0pt] (4-) {};
\draw (5,0) node[draw,circle,fill=white,minimum size=5pt,inner sep=0pt] (5-) {};
\draw (6,0) node[draw,circle,fill=gray,minimum size=5pt,inner sep=0pt] (6-) {};
\draw (7,0) node[draw,circle,fill=gray,minimum size=5pt,inner sep=0pt] (7-) {};

\draw (1+) to (6+);
\draw (1-) to (4-);

\draw [dashed] (1+) to [bend left=60] (6+);
\draw [dashed] (2+) to [bend left=60] (5+);
\draw [dashed] (3+) to [bend left=60] (4+);
\draw [dashed] (1-) to [bend right=60] (4-);
\draw [dashed] (2-) to [bend right=60] (3-);

\node at (.5,-.5) {$D'$};
\node at (.5,1.75) {$A'$};
\node at (3.5,1.75) {$C'$};
\node at (6.5,1.75) {$B'$};
\node at (1,1) {$w_1^+$};
\node at (2,1) {$w_2^+$};
\node at (3,1) {$w_3^+$};
\node at (4,1) {$w_4^+$};
\node at (5,1) {$w_5^+$};
\node at (6,1) {$w_6^+$};
\node at (1,.35) {$w_1^-$};
\node at (2,.35) {$w_2^-$};
\node at (3,.35) {$w_3^-$};
\node at (4,.35) {$w_4^-$};

;\end{tikzpicture}
\]
\caption{Rotation Expansion applied to $\mathcal{M}$ with $a_1 = 5$ and $b_1=3$ (left) to obtain $\mathcal{M}'$ (right)}
\label{fig:InductionRotationExpansion}
\end{figure}
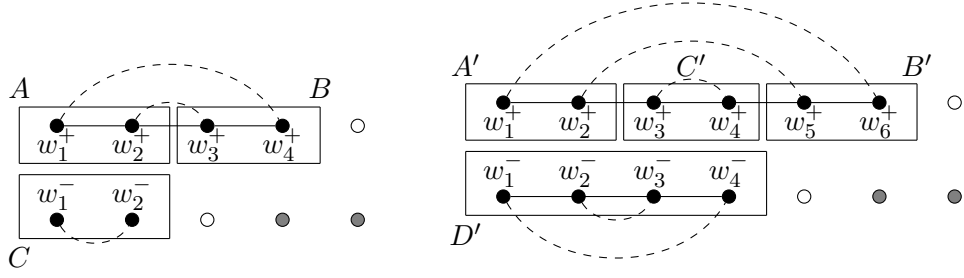

We have 
\begin{eqnarray}
\label{ind2.1} \mathcal{E}(D') &=& \mathcal{E}(A \cup B),\\
\label{ind2.2} \mathcal{E}(C') &=& -\mathcal{E}(C). 
\end{eqnarray} 


\noindent 
By equation (\ref{ind2.1}), $\mathcal{E}(\sigma_2)$ is unbroken.  Without loss of generality, 
\begin{eqnarray}
\label{ind2.3} \mathcal{E}(A' \cup C') &=& -\mathcal{E}(A \cup B),
\end{eqnarray} 
so $\mathcal{E}(A' \cup C') \subseteq \mathcal{E}(\sigma_1)$ is unbroken.  All other eigenvalues in $\mathcal{E}(\sigma_1)$ are symmetric to eigenvalues in $\mathcal{E}(A' \cup C')$, and consequently, $\mathcal{E}(\sigma_1)$ is unbroken.  

\begin{case}\label{unbroken3}
Pure Expansion:
\end{case}

\begin{figure}[H]
\[\begin{tikzpicture}
[decoration={markings,mark=at position 0.6 with 
{\arrow{angle 90}{>}}}]

\draw (.5,1) -- (3.5,1);
\draw (.5,.25) -- (3.5,.25);
\draw (.5,.25) -- (.5,1);
\draw (3.5,.25) -- (3.5,1);

\draw (4.5,1) -- (6.5,1);
\draw (4.5,.25) -- (6.5,.25);
\draw (4.5,.25) -- (4.5,1);
\draw (6.5,.25) -- (6.5,1);

\draw (1,.75) node[draw,circle,fill=black,minimum size=5pt,inner sep=0pt] (1+) {};
\draw (2,.75) node[draw,circle,fill=black,minimum size=5pt,inner sep=0pt] (2+) {};
\draw (3,.75) node[draw,circle,fill=black,minimum size=5pt,inner sep=0pt] (3+) {};
\draw (4,.75) node[draw,circle,fill=white,minimum size=5pt,inner sep=0pt] (4+) {};
\draw (5,.75) node[draw,circle,fill=black,minimum size=5pt,inner sep=0pt] (5+) {};
\draw (6,.75) node[draw,circle,fill=black,minimum size=5pt,inner sep=0pt] (6+) {};
\draw (7,.75) node[draw,circle,fill=white,minimum size=5pt,inner sep=0pt] (7+) {};

\draw (1,0) node[draw,circle,fill=gray,minimum size=5pt,inner sep=0pt] (1-) {};
\draw (2,0) node[draw,circle,fill=gray,minimum size=5pt,inner sep=0pt] (2-) {};
\draw (3,0) node[draw,circle,fill=gray,minimum size=5pt,inner sep=0pt] (3-) {};
\draw (4,0) node[draw,circle,fill=gray,minimum size=5pt,inner sep=0pt] (4-) {};
\draw (5,0) node[draw,circle,fill=gray,minimum size=5pt,inner sep=0pt] (5-) {};
\draw (6,0) node[draw,circle,fill=gray,minimum size=5pt,inner sep=0pt] (6-) {};
\draw (7,0) node[draw,circle,fill=gray,minimum size=5pt,inner sep=0pt] (7-) {};

\draw (1+) to (2+);
\draw (2+) to (3+);
\draw (5+) to (6+);

\draw [dashed] (1+) to [bend left=60] (3+);
\draw [dashed] (5+) to [bend left=60] (6+);

\node at (.5,1.25) {$A$};
\node at (4.5,1.25) {$B$};
\node at (1,.5) {$w_1^+$};
\node at (2,.5) {$w_2^+$};
\node at (3,.5) {$w_3^+$};
\node at (4,.5) {$w_4^+$};
\node at (5,.5) {$w_5^+$};
\node at (6,.5) {$w_6^+$};
\node at (7,.5) {$w_7^+$};

;\end{tikzpicture}\]

\[\begin{tikzpicture}
[decoration={markings,mark=at position 0.6 with 
{\arrow{angle 90}{>}}}]

\draw (.5,1.5) -- (2.5,1.5);
\draw (.5,.75) -- (2.5,.75);
\draw (.5,.75) -- (.5,1.5);
\draw (2.5,.75) -- (2.5,1.5);

\draw (2.6,1.5) -- (3.4,1.5);
\draw (2.6,.75) -- (3.4,.75);
\draw (2.6,.75) -- (2.6,1.5);
\draw (3.4,.75) -- (3.4,1.5);

\draw (3.5,1.5) -- (6.5,1.5);
\draw (3.5,.75) -- (6.5,.75);
\draw (3.5,.75) -- (3.5,1.5);
\draw (6.5,.75) -- (6.5,1.5);

\draw (6.6,1.5) -- (7.4,1.5);
\draw (6.6,.75) -- (7.4,.75);
\draw (6.6,.75) -- (6.6,1.5);
\draw (7.4,.75) -- (7.4,1.5);

\draw (7.5,1.5) -- (9.5,1.5);
\draw (7.5,.75) -- (9.5,.75);
\draw (7.5,.75) -- (7.5,1.5);
\draw (9.5,.75) -- (9.5,1.5);

\draw (.5,-.25) -- (2.5,-.25);
\draw (.5,.6) -- (2.5,.6);
\draw (.5,.6) -- (.5,-.25);
\draw (2.5,.6) -- (2.5,-.25);

\draw (1,1.25) node[draw,circle,fill=black,minimum size=5pt,inner sep=0pt] (1+) {};
\draw (2,1.25) node[draw,circle,fill=black,minimum size=5pt,inner sep=0pt] (2+) {};
\draw (3,1.25) node[draw,circle,fill=black,minimum size=5pt,inner sep=0pt] (3+) {};
\draw (4,1.25) node[draw,circle,fill=black,minimum size=5pt,inner sep=0pt] (4+) {};
\draw (5,1.25) node[draw,circle,fill=black,minimum size=5pt,inner sep=0pt] (5+) {};
\draw (6,1.25) node[draw,circle,fill=black,minimum size=5pt,inner sep=0pt] (6+) {};
\draw (7,1.25) node[draw,circle,fill=black,minimum size=5pt,inner sep=0pt] (7+) {};
\draw (8,1.25) node[draw,circle,fill=black,minimum size=5pt,inner sep=0pt] (8+) {};
\draw (9,1.25) node[draw,circle,fill=black,minimum size=5pt,inner sep=0pt] (9+) {};
\draw (10,1.25) node[draw,circle,fill=white,minimum size=5pt,inner sep=0pt] (10+) {};

\draw (1,0) node[draw,circle,fill=black,minimum size=5pt,inner sep=0pt] (1-) {};
\draw (2,0) node[draw,circle,fill=black,minimum size=5pt,inner sep=0pt] (2-) {};
\draw (3,0) node[draw,circle,fill=white,minimum size=5pt,inner sep=0pt] (3-) {};
\draw (4,0) node[draw,circle,fill=gray,minimum size=5pt,inner sep=0pt] (4-) {};
\draw (5,0) node[draw,circle,fill=gray,minimum size=5pt,inner sep=0pt] (5-) {};
\draw (6,0) node[draw,circle,fill=gray,minimum size=5pt,inner sep=0pt] (6-) {};
\draw (7,0) node[draw,circle,fill=gray,minimum size=5pt,inner sep=0pt] (7-) {};
\draw (8,0) node[draw,circle,fill=gray,minimum size=5pt,inner sep=0pt] (8-) {};
\draw (9,0) node[draw,circle,fill=gray,minimum size=5pt,inner sep=0pt] (9-) {};
\draw (10,0) node[draw,circle,fill=gray,minimum size=5pt,inner sep=0pt] (10-) {};

\draw (1+) to (2+);
\draw (2+) to (3+);
\draw (3+) to (4+);
\draw (4+) to (5+);
\draw (5+) to (6+);
\draw (6+) to (7+);
\draw (7+) to (8+);
\draw (8+) to (9+);
\draw (1-) to (2-);

\draw [dashed] (1+) to [bend left=60] (9+);
\draw [dashed] (2+) to [bend left=60] (8+);
\draw [dashed] (3+) to [bend left=60] (7+);
\draw [dashed] (4+) to [bend left=90] (6+);
\draw [dashed] (1-) to [bend right=60] (2-);

\node at (.5,-.5) {$D'$};
\node at (.5,1.75) {$B'$};
\node at (3,1.75) {$E'$};
\node at (5,1.75) {$A'$};
\node at (7,1.75) {$F'$};
\node at (9.5,1.75) {$C'$};
\node at (1,1) {$w_1^+$};
\node at (2,1) {$w_2^+$};
\node at (3,1) {$w_3^+$};
\node at (4,1) {$w_4^+$};
\node at (5,1) {$w_5^+$};
\node at (6,1) {$w_6^+$};
\node at (7,1) {$w_7^+$};
\node at (8,1) {$w_8^+$};
\node at (9,1) {$w_9^+$};
\node at (1,.35) {$w_1^-$};
\node at (2,.35) {$w_2^-$};

;\end{tikzpicture}
\]
\caption{Pure Expansion applied to $\mathcal{M}$ with $a_1 = 4$ and $a_2=3$ (top) to obtain $\mathcal{M}'$ (bottom)}
\label{fig:InductionPureExpansion}
\end{figure}
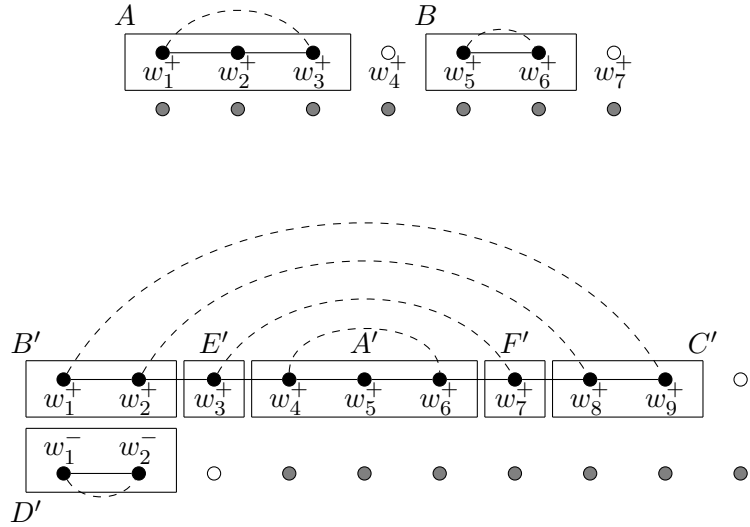

We have 

\begin{eqnarray}
\label{ind2.4} \mathcal{E}(B') &=& -\mathcal{E}(B), \\
\label{ind2.5} \mathcal{E}(A') &=& \mathcal{E}(A), \\
\label{ind2.6} \mathcal{E}(C') &=& \mathcal{E}(B), \\
\label{ind2.7} \mathcal{E}(D') &=& \mathcal{E}(B).  
\end{eqnarray} 
\noindent 
By equation (\ref{ind2.7}), $\mathcal{E}(\sigma_2)$ is unbroken.  
Let $\gamma = \alpha_{a_2}(\widehat{F})$ for $\alpha_{a_2} \in \pi_1$.  By Lemma \ref{lem:simple}, $\gamma = 1, 2,$ or $3$.  Since 
\[\alpha_{1}(\widehat{F})+ \alpha_{2}(\widehat{F}) + ... + \alpha_{a_2-1}(\widehat{F}) = -1,\] 
and $\gamma = 1, 2,$ or $3$, $\mathcal{E}(B' \cup E')$ contains $0, 1,$ or $2$.  In particular, 
\[\alpha_{1}(\widehat{F})+ \alpha_{2}(\widehat{F}) + ... + \alpha_{a_2-1}(\widehat{F}) + \alpha_{a_2}(\widehat{F})= \gamma-1.\] 
It is clear that if $a_2$ is even, then the eigenvalues in $\mathcal{E}(B' \cup E') \setminus \mathcal{E}(B')$ are 
\[\{ \gamma, 
\gamma + \alpha_{a_2-1}(\widehat{F}), 
\gamma + \alpha_{a_2-2}(\widehat{F}), ..., 
\gamma + \alpha_{\frac{a_2}{2}+1}(\widehat{F}), 
\gamma - 1, 
\gamma - 1 +  \alpha_{a_2-1}(\widehat{F}), 
\gamma - 1 + \alpha_{a_2-2}(\widehat{F}), ..., 
\gamma - 1 + \alpha_{\frac{a_2}{2}+1}(\widehat{F})
\}.\] 
By Lemma \ref{lem:simple}, $|\alpha_i(\widehat{F})| = 0, 1, 2,$ or $3$ for all $i$.  

If $|\alpha_i(\widehat{F})| \neq 3$ for some $i \in \{a_2 - 1, a_2 - 2, ..., \frac{a_2}{2}+1\}$, then $\mathcal{E}(B' \cup E') \setminus \mathcal{E}(B')$ is unbroken.  Moreover, since 
\[\alpha_{a_2+1}(\widehat{F})+ \alpha_{a_2+2}(\widehat{F}) + ... + \alpha_{a_1+a_2-1}(\widehat{F}) = 1,\] 
the multiset
\[\mathcal{E}(B' \cup E' \cup A') \setminus \left[\mathcal{E}(B' \cup E') ~\cup~ \mathcal{E}(A')\right]\] 
is unbroken. 
But since $\{0, 1\} \subseteq \mathcal{E}(A')$, the multiset 
$\mathcal{E}(B' \cup E' \cup A')$ is unbroken and contains $0$.  Similarly, the multiset 
$\mathcal{E}(A' \cup F' \cup C')$ is unbroken and contains $0$.  Therefore,  
$\mathcal{E}(B' \cup E' \cup A') \cup \mathcal{E}(A' \cup F' \cup C')$ is unbroken and contains $0$.  Any remaining eigenvalue in $\mathcal{E}(\sigma_1)$ is symmetric to an eigenvalue in $\mathcal{E}(B' \cup E' \cup A') \cup \mathcal{E}(A' \cup F' \cup C')$.  Therefore, $\mathcal{E}(\sigma_1)$ is unbroken.  

If $|\alpha_i(\widehat{F})| = 3$ for all $i \in \{a_2 - 1, a_2 - 2, ..., \frac{a_2}{2}+1\}$, then $\mathcal{E}(B' \cup E') \setminus \mathcal{E}(B')$ is not necessarily unbroken.  Moreover, any numbers preventing $\mathcal{E}(B' \cup E') \setminus \mathcal{E}(B')$ from being unbroken must be congruent $(\bmod ~3)$.  
Let $x$ be any such number.  Observe that 
\[\mathcal{E}(B' \cup E') \setminus \mathcal{E}(B') = 
-[\mathcal{E}(C' \cup F') \setminus \mathcal{E}(C')].\]
Then $x$ is symmetric to some number in $\mathcal{E}(C' \cup F') \setminus \mathcal{E}(C')$ and exists somewhere in $\mathcal{E}(\sigma_1)$.  The rest of the argument follows similarly to the previous argument: that is,  when $|\alpha_i(\widehat{F})| \neq 3$ for each $i \in \{a_2 - 1, a_2 - 2, ..., \frac{a_2}{2}+1\}$.  

If $a_2$ is odd, then the eigenvalues in $\mathcal{E}(B' \cup E') \setminus \mathcal{E}(B')$ are the same as in the previous case in addition to $\gamma + \alpha_{\lf\frac{a_2}{2}\rf}(\widehat{F})$.  But since 
\[\alpha_{a_2+1}(\widehat{F})+ \alpha_{a_2+2}(\widehat{F}) + ... + \alpha_{a_1+a_2-1}(\widehat{F}) = 1,\] 
the multiset
\[\mathcal{E}(B' \cup E' \cup A') - \left[\mathcal{E}(B' \cup E') ~\cup~ \mathcal{E}(A')\right]\] 
is unbroken.  That $\mathcal{E}(\sigma_1)$ is unbroken follows from the argument where $a_2$ is even.  
\end{proof}

\begin{remark}
Any type-A orbit meander consists entirely of maximally connected components of type A.  So the proof of Theorem \ref{thm:symmetry} shows that the spectrum of any Frobenius type-A seaweed forms a symmetric distribution.  The type-A unbroken spectrum result follows with the same induction proof, but with a different inductive base.  That is, any Frobenius type-A
orbit meander is the result of a sequence of moves in Lemma \ref{Expansion} starting from $\mathcal{M}_1(1 \dd \emptyset)$.  We note that the base has spectrum $\{0,1\}$, which is unbroken.  
\end{remark}

\begin{remark}
The proof of Theorem \ref{thm:unbroken}, in particular, the proof for Case \ref{unbroken3}, would have held even if the simple eigenvalues from Lemma \ref{lem:simple} were bounded in absolute value by $4$.  The type-A proof by Coll et al. in \textbf{\cite{Coll typea}} requires the bound to be $3$.  
\end{remark}

\section{Afterword}
In a forthcoming article, we show that Theorem \ref{thm:main} is also true for type-D Frobenius seaweeds.  The type-D case is complicated by the bifurcation point in the type-D Dynkin diagram, and requires several  different induction bases.  Moreover, 
the computation of the simple eigenvalues 
in type-D case is both different and more involved than in the type B and C cases.

\end{document}